\documentclass[12pt]{article}
\newcommand{\average}{-\!\!\!\!\!\!\int}

\newcommand{\comment}[1]{}

\renewcommand{\div}{{\mathop{\rm\,div\,}\nolimits}}

\newcommand{\dist}{\mathop{\rm dist}\nolimits}

\providecommand{\qed}{\vrule height 6pt depth 0pt width 3 pt}

\newcommand{\reals}{{\bf R}}

\newcommand{\BibTeX}{{\rm B\kern-.05em{\sc i\kern-.025em b}\kern-.08em     
    T\kern-.1667em\lower.7ex\hbox{E}\kern-.125emX}}

\usepackage{hyperref}

\newcommand{\note}[1]{}

\renewcommand\marginpar[1]{}

\providecommand\lp[1]{L^{#1}}

\usepackage{amsmath} 
\usepackage{amssymb}

\newcommand{\sobolev}[2]{ \dot W ^ {{#2},{#1}}}
\newcommand{\ball}[2]{B_{#2}(#1)}

\newcommand{\locdom}[2]{{\Omega_{#2}(#1)}}
\newcommand{\sball}[2]{{\Delta_{#2}(#1)}}
\newcommand{\cyl}[2]{Z_{#2}(#1)}
\newcommand{\avg}[3]{ \bar #1 _{#2,#3}}
\newcommand{\sign}{\mathop {\rm sign}\nolimits}
\newenvironment{proof}[1][Proof]{\begin{trivlist}\item[\hskip \labelsep
{\it #1. }]}{\hfill \qed \goodbreak \end{trivlist}}   

\newenvironment{remark}[1][Remark]{\begin{trivlist}\item[\hskip \labelsep
{\it #1. }]}{  \goodbreak \end{trivlist}}   

\numberwithin{equation}{section} 
\newtheorem{theorem}[equation]{Theorem}
\newtheorem{proposition}[equation]{Proposition}
\newtheorem{corollary}[equation]{Corollary}

\usepackage{amssymb}

\begin{document}
\title{The Green function  for elliptic systems in two dimensions}

\author{J.L. Taylor \\ Department of Mathematics \\ Murray State University \\
Murray, Kentucky  42071, USA
\and
 S. Kim
\footnote{
Seick Kim is supported by NRF Grant No. 2010-0008224 and R31-10049 (WCU program).}
  \\ Department of Computational Science and Engineering \\ Yonsei University \\ Seoul, Korea
\and
 R.M. Brown\footnote{
This work was partially supported by a grant from the Simons
Foundation (\#195075 to Russell Brown).}
 \\ Department of Mathematics\\ University of Kentucky \\
Lexington, Kentucky 40506, USA}
\date{}
\maketitle

\abstract{ 
We construct the fundamental solution or Green function for a divergence
form elliptic system in two dimensions with bounded and measurable coefficients.
We consider the 
elliptic system  in a Lipschitz domain with mixed boundary conditions. Thus we  
 specify Dirichlet data
on part of the boundary and Neumann data on the remainder of the boundary. 
We require
a corkscrew or non-tangential accessibility condition on the set where we specify
Dirichlet boundary conditions. 
Our proof
proceeds by defining a variant of the space $BMO (\Omega)$ that is adapted to the
boundary conditions and showing that the solution exists in this 
space. 
We also give a  construction of  the Green function with 
Neumann boundary conditions and the fundamental solution in the plane. 
}

\section{Introduction}\label{Introduction}
We consider a weak formulation of the mixed problem for a second-order
elliptic operator in a bounded, connected, and  open set $ \Omega$ in $ \reals ^2$. To
state the mixed problem, we fix a decomposition of the boundary $
\partial \Omega = D \cup N$ with $ D\cap N= \emptyset$. We let $L$ be
an elliptic operator  in divergence form with  bounded and measurable
coefficients and we  consider the boundary value problem
\begin{equation}\label{MP}
\left\{
\begin{array}{ll}
Lu =f  \qquad &\mbox{in } \Omega \\
u= 0  \qquad & \mbox{on } D\\
\frac{\partial u }{ \partial \nu} = f _N , \qquad & \mbox {on }N.
\end{array}
\right. 
\end{equation}
where $ \partial /\partial \nu$ is the natural co-normal derivative
associated with the operator $L$. We will require that $  \Omega$ be a
Lipschitz domain and that the set $D$ satisfy a cork-screw condition
(or non-tangential accessibility condition) as in \cite{TOB:2011}.
We will give a precise formulation of the mixed problem 
in section \ref{Domains}. We emphasize that our results apply to 
elliptic systems where  $L$ acts on vector-valued functions as
well as equations.  Our goal is to give a construction of the
matrix-valued 
Green function for this boundary condition and show that the Green
function satisfies the estimates
\begin{eqnarray*}
|G(x,y) | & \leq & C( 1+ \log ( d /|x-y|)), \qquad x, y \in \Omega \\
|G(x,y) - G(x,z) | & \leq & C \frac { |y-z|^ \gamma }{ |x-y | ^
\gamma } , \qquad x,y,z \in \Omega , \ |y-z|< \frac 1 2 |x-y|\\
|G(x,y) | & \leq & \frac { \dist (y, D )^ \gamma}{ |x-y |^ \gamma } , \qquad
\dist (y, D) < \frac 1 2 |x-y | . 
\end{eqnarray*}
where $d$ is the diameter of $ \Omega$ and the constant $C$ and $
\gamma >0$ depend only on assumptions on the operator, the domain, and
the decomposition of the boundary.

Given the logarithmic singularity of the Green function, it is natural
to look for a Green function in the space of functions of bounded mean
oscillation. We fix  a set $D$ and  define the  space $BMO_D(
\Omega)$ which consists of functions in $BMO( \Omega)$ and  which
vanish in an appropriate sense as we approach $D$. We will construct
the Green function in this space and then  obtain the pointwise
estimates.

There is a great deal of literature on the existence of Green
functions and we do not try to summarize it all here.
 Littman, Stampacchia, and Weinberger \cite{LSW:1963}
establish that the Green function has a logarithmic singularity in two
dimensions. Gr\"uter and Widman \cite{MR657523} give a nice
construction of the Green function for the Dirichlet problem in
dimensions three and higher and treat operators that are not
self-adjoint.  Kenig and Pipher \cite{KP:1993} give a construction of
the Neumann function or the Green function for the Neumann problem
when $ n \geq 3$.  The existence of a global fundamental solution for
an elliptic operator in the plane was established by Kenig and Ni
\cite{KN:1985}. Using their results, Chanillo and Li \cite{MR1190215}
show that the global fundamental solution lies in $BMO( \reals ^2)$.

The results of the previous paragraph
 are for single equations. More recently, there has been
interest in constructing Green functions for elliptic systems. In this
case, we are only able to obtain upper bounds.  The work of Auscher
and collaborators \cite{MR1600066} establishes the existence of a
Green function for operators with complex coefficients (which may be
viewed as a system of equations with real coefficients).  This work
was motivated in part by their interest in the Kato problem.  Dolzmann
and M\"uller \cite{MR1354111} construct the Green function in a $C^1$-domain 
for an elliptic operator with continuous coefficients.  Dong and
Kim \cite{MR2485428} establish the existence of a Green function for
elliptic systems which are not assumed to be self-adjoint under the
hypothesis that solutions of the operator are bounded or H\"older
continuous. Their argument gives Gaussian upper bounds for the
parabolic Green function and then integrates the parabolic Green
function in time giving the Green function for the elliptic problem.
Recently, Choi and Kim \cite{arXiv:1112.2436v1} have given a
construction of the Neumann function in dimensions three and higher
under the assumption that solutions of the operator satisfy local
H\"older continuity and boundedness estimates.  In two dimensions, the
necessary results on the H\"older continuity of solutions dates back
to Morrey \cite{MR1501936}.  A paper of Calanchi, Rodino, and Tri
\cite{MR1471017} observes that the Green function for the Dirichlet
problem for the Laplacian lies in $BMO$ for planar domains. Their
proof relies on the maximum principle which is not available to treat
the systems considered here.  In addition to our main goal of treating
the Green function for the mixed problem, the present paper
complements the work of Choi and Kim by providing a construction of
the Green function for the Neumann problem in two dimensions.  Recent
work of D.~Mitrea and I.~Mitrea \cite{MR2763343} gives the existence
of the Dirichlet Green function for various constant coefficient
operators in Lipschitz domains.

The present paper was motivated by an effort to construct the Green
function for elliptic systems under mixed boundary conditions. Taylor,
Ott, and Brown \cite{TOB:2011} give a construction of the Green
function for a class of mixed problems for the Laplacian in a
Lipschitz domain in dimensions two and higher.  Their
argument begins with the free space fundamental solution and uses a 
reflection to
first construct a Green function for Neumann boundary conditions and finally
corrects the boundary conditions  to give
a Green function for mixed boundary conditions.   Below,
we provide a new proof of their results in two dimensions and
establish the existence of a Green function with mixed boundary
conditions for a large class of
elliptic systems in two dimensions. We expect that the Green 
function constructed below  
will be useful in extending Taylor, Ott, and Brown's
results on the $L^p$-mixed problem for the Laplacian
 to systems in two dimensions. It is an interesting, open question
to study the mixed problem for elliptic systems 
in dimensions three and higher. The main difficulty in carrying out
this extension is the lack of estimates for the Green function. 

The construction of the Green function for an operator $L$ is closely
related to local, scale-invariant estimates for solutions of $L$. In
the case of a single equation, the H\"older continuity of solutions to
the mixed problem was established by Stampacchia \cite{GS:1960} using
the method of De Giorgi. Stampacchia considers the mixed problem in
domains which locally are equivalent by a bi-Lipschitz transformation
to a mixed problem in a half-space with the boundary between $D$ and
$N$ a hyperplane.  A more recent work of Gr\"oger \cite{MR990595} uses the
method of Meyers \cite{NM:1963} to show 
that solutions of the mixed problem with
nice data satisfy $ \nabla u \in L^ { 2+ \epsilon } ( \Omega)$.  
Gr\"oger's assumptions on
the domain and the decomposition of the boundary are similar to those
of Stampacchia.  When considering the more general decompositions of
the boundary introduced by Taylor, Ott, and Brown and considered in
this paper, it seems to be simpler to use the method of reverse
H\"older inequalities as in Gehring \cite{FG:1973} and Giaquinta and
Modica \cite{MR549962} to obtain that $ \nabla u \in 
\lp  { 2+ \epsilon}( \Omega)$.

Our formulation of the mixed problem allows for the extreme cases where
$D = \emptyset $ or $D = \partial \Omega$ which give the Neumann problem and the
Dirichlet problem, respectively.  The Green function for the Dirichlet problem is
treated alongside the mixed problem in section \ref{Green}.  
The properties of 
the Green function are given in Theorem \ref{GreenTheorem}. 
 The  changes  needed for
the Neumann problem  appear in section \ref{Neumann} and the properties of the
Green function for the Neumann problem are given in Theorem \ref{NGreenTheorem}. 
 In Section \ref{Free}, we provide a 
proof of the existence of a fundamental solution in the plane. 
The construction of a global fundamental solution for an equation was perhaps
known, however the detailed construction appears to have first been written
down  in a paper of 
Kenig and Ni from 1985 \cite{KN:1985}. We provide a different construction
of the fundamental solution which also applies to 
 systems.     The existence of a fundamental 
solution in the plane and the properties of this fundamental solution are given in
Theorem \ref{FGreenTheorem}. 
The Lam\'e system (with variable Lam\'e parameters) provides a family of examples
to illustrate the use of our results.  We describe this system in section 
\ref{Domains} and 
show that it satisfies the hypotheses of our main results.

\section{Preliminary material and the weak formulation of the
  mixed problem}
 \label{Domains}
 
We will consider the Green function for boundary value problems in a 
 bounded Lipschitz domain $ \Omega$  and we begin by 
giving the definition of these domains. 
Given a constant $M>0$,   
$ x \in \partial \Omega$, and $ r>0$,  
we let $ \cyl x r = \{ y : |y_1- x _1  | <r ,
|y_2 -x _2 | < ( 3M +1)r\}   $. We say that $ \cyl x r $ is a {\em coordinate
rectangle for $ \Omega$ }if there is a Lipschitz function $ \phi : \reals \rightarrow
\reals$ so that 
\begin{eqnarray*}
\Omega \cap \cyl x r&  = & \{ y : y_ 2 > \phi ( y_1) \} \cap \cyl x r
\\
\partial \Omega \cap \cyl x r&  = & \{ y : y_ 2 = \phi ( y_1) \} \cap
\cyl x r  
\end{eqnarray*}
We assume that the coordinate system used in this coordinate rectangle is
a rotation of the standard coordinate system. 
We say that $ \Omega $ is a {\em Lipschitz domain }if $ \Omega$ is a 
bounded  connected open set and 
 for each $ x \in \partial \Omega$, we may find a coordinate 
rectangle centered at $x$. Since the boundary is compact, 
we may cover $ \partial \Omega$ by a finite collection of coordinate rectangles 
$\{ \cyl { x_i } { r _ i } \}_{ i =1 } ^ N$ so that 
 each $ \cyl {x_i}  { 100 r_i } $  is also
a coordinate rectangle. We set
 $ r_0 = \min \{ r _ i : i = 1, \dots, N\}$ and we will use this 
value as a characteristic length of the domain when 
we state scale-invariant estimates. 

For $ x\in \bar \Omega$ and $ \rho \in ( 0, 4r_0)$, 
we define{ \em local domains }$ \locdom x 
\rho $ which will play the role of disks in our work. In the case that $\dist ( x,
\partial \Omega ) > \rho$, we let $ \locdom x \rho = \ball x\rho$, 
the disk centered at $x$ with radius 
$ \rho$. In the case that $ \dist (x , \partial \Omega ) \leq \rho$, 
we pick a coordinate 
rectangle  containing $ x  = (x_1, x_2)$ and using the coordinates 
for this rectangle, we  let $ 
\hat x  = ( x_1 , \phi(x_1))$ be obtained by projecting 
onto the boundary and  define  $ \locdom x 
\rho = \cyl { \hat x } \rho \cap \Omega$. 

\note{ If $ x^* = x+3Mre_2$ and $y$ is in the disk 
centered at $x^*$ of radius $r/2$ and $z$ 
$\partial \Omega \cap \cyl x r$, then the slope of the line 
joining $z$ to $y$ is bounded below 
by $4M/3$. This implies the domain is star-shaped with respect to $y$. }

For $ x\in \partial \Omega$ and $ \rho  \in (0, r_0)$, 
we define a {\em boundary interval }$ \sball x \rho$ by 
 $ \sball x \rho =\cyl x \rho \cap \partial \Omega$. 
The domains $ \locdom x  \rho$ are star-shaped Lipschitz domains and the 
boundary intervals $ \sball x \rho$ are connected.
These properties are helpful in establishing the  Poincar\'e and 
Korn inequalities introduced below.  For this reason, we 
prefer them to $\ball x \rho \cap \Omega$ 
and $ \ball x \rho\cap \partial \Omega$ used in \cite{OB:2009,TOB:2011},
  for example. There is a price to pay as given $ x$ and $\rho$ 
  we will have several choices for $ \sball x \rho$ 
and $ \locdom x\rho$. Our results will hold 
for any such choice provided that when several of 
these objects occur in an estimate we use the 
same coordinate rectangle to define each of them.

We let  $ D\subset \partial \Omega $ be the set where we specify the Dirichlet data 
and then we put $ N = \partial \Omega \setminus D$. We will require that $D$ 
satisfy the {\em interior corkscrew condition}. 
This means that for each $ x\in \partial D$ (where the 
boundary is taken with respect 
to $\partial \Omega$) and $ r \in (0, r_0)$, we may find  $ x_r \in D$ 
so that 
\begin{equation} 
\label{Corkscrew}
 |x- x_r | \leq r \mbox{ and } \dist (x_r, N) \geq M^ { -1 } r. 
\end{equation}
As our boundary is locally a Lipschitz graph, this easily 
implies that for each $x \in D$ and $ r \in (0, r_0)$, 
we have that $ \sigma ( \sball x r  \cap D ) \geq c r$ (see \cite[Section 1]
{TOB:2011} for details). 
Note that if $D$ is not empty, then the corkscrew condition implies 
that the interior of $D$ is not empty.

Now we turn to the precise description of the boundary value problem. 
As we are considering an operator acting on vector-valued functions,
 most of the functions we consider will take values in $ \reals ^m$ for some $m$.
  We do not explicitly denote the 
target in our notation for function spaces. However, we emphasize that in the
 definitions below, all of the functions will be vector-valued. 
 We let $ \sobolev p1( \Omega)$ denote the homogeneous Sobolev space
 of functions having one derivative in $L^p ( \Omega)$ with the norm
 $ \|u \|_ { \sobolev p1 _D ( \Omega)} = \| \nabla u \| _ { L^ p (\Omega)}$. 
Given a set   $D\subset \partial 
\Omega$, we  let $ \sobolev p 1 _D ( \Omega)$ denote the subspace 
obtained by taking the closure 
in $ \sobolev p 1 ( \Omega)$ 
of the smooth functions in $ \bar \Omega$ which vanish in a neighborhood of $D$.
Thus, the elements
in $ \sobolev p 1 _D ( \Omega)$ are functions which in some sense vanish on $D$. 
When $D$ is not empty and satisfies the corkscrew condition and  $1< p< \infty$,
 we have that if $ u \in \sobolev  p 1 _D ( \Omega)$, then 
\begin{equation}
\label{DMatters}
r_0 ^ { -1/p } \left ( \int _ { \partial \Omega } |u|^ p \, d\sigma \right) ^ { 1/p } +
r_0 ^ { -2/p } \left( \int _ \Omega |u|^p \, dy\right ) ^ { 1/p} 
\leq C \left( \int _ { \Omega } |\nabla u |^ 2 \, dy \right ) ^ { 1/2} . 
\end{equation}
for a constant $C$ which depends on $M$, $p$, and  $\Omega$.  
We will let $ \sobolev  { t'}   {-1} _D ( \Omega)$ be the dual of
 $ \sobolev t 1 _D ( \Omega)$ when $ 1\leq t < \infty$. 
The space  $ \sobolev 2 { 1/2} _D (\partial \Omega) $ is defined to  be the 
 image of $ \sobolev 2 1 _D ( \Omega)$ 
under the trace map and then $ \sobolev 2 { -1/2} _D ( \partial \Omega)$ 
the dual of $ \sobolev 2 {1/2} _D (\partial  \Omega)$.  We let 
$ \langle \cdot, \cdot \rangle: 
\sobolev 2 {-1}  _D ( \Omega) \times \sobolev 2 1 _D( \Omega) 
\rightarrow \reals$ and 
 $ \langle \cdot, \cdot \rangle_{\partial \Omega}: 
\sobolev 2 {-1/2}  _D (\partial \Omega) \times \sobolev 2 {1/2} _D( \partial \Omega) 
\rightarrow \reals$
be the pairings of duality. 

\note{ Do we need the boundary term in (\ref{DMatters}). }

The operator $L$ will act on vector-valued functions in the plane. 
Formally, $L$ is given by 
\begin{equation}
\label{OpDef}
(L u )^ \alpha = \sum _ { i,j= 1} ^ 2
\sum _ { \beta= 1} ^ m \frac{\partial}{\partial x _ i 	}
a_ { \alpha \beta } ^ { ij} 
\frac{\partial u ^ \beta}{\partial x _j }, \qquad \alpha = 1, \dots, m.
\end{equation}
The coefficients $a^{ij}_{\alpha \beta}$ 
are assumed to be real, bounded,  and measurable functions
\begin{equation}
\label{CoeffBound}
\max \{ \| a  ^{ij} _{\alpha \beta }  \|_ { L^ \infty ( \Omega ) } : 
i,j = 1,2,  \alpha, \beta = 1, \dots, m \} \leq M . 
\end{equation} 
We do not assume a symmetry condition on the coefficients and thus the
operator $L$ will not be self-adjoint. 
We will
define $Lu$ as an element of $ \sobolev 2 { -1} _D( \Omega)$ 
via the quadratic form 
$$
A( u, \phi ) 
= \int _ { \Omega } a^ { ij } _ { \alpha \beta } 
\frac { \partial u ^ \beta } { \partial x _ j} 
\frac{ \partial \phi ^ \alpha } { \partial x_i} \, dx .
$$
Here and throughout this paper, we follow the convention that we
sum on repeated indices. 
Since the coefficients are bounded,  for some constant $C=C(M,m)$ we have
\begin{equation}
\label{Bounded}
|A ( u, \phi )| \leq C \|u \| _ { \sobolev 21  ( \Omega) }
 \| \phi \|_ { \sobolev 2 1 ( \Omega)} 
.
\end{equation}
We assume the following local ellipticity condition. If 
$ x \in \bar \Omega$,   $ \rho \in  (0, r_0)$,  
and $u \in \sobolev 2 1 _ \emptyset ( \Omega)$, then for 
all constant vectors $ c \in \reals ^ m$, we have 
\begin{equation} 
\label{Ellipticity}
M^ { -1} \int _ { \locdom x \rho } |\nabla u |^ 2 \, dy 
\leq \int _ { \locdom x \rho } a^ { ij}_ { \alpha \beta } 
\frac { \partial u ^ \beta} { \partial y_j } \frac {  \partial u ^ \alpha }
{ \partial y _i }  + \rho ^ { -2} |u - c|^2\, dy . 
\end{equation}
\note{ Note that if this estimate holds for one value of $c$, it 
holds for all values of $c$.  For strongly elliptic systems, 
we do not need the $L^2$ norm. For 
Lame, we need the extra term. 
} 
In addition, we assume a global coercivity condition  on $D$ and the
form $ A$ 
\begin{equation}\label{Coerce}
M^ { -1} \int _ { \Omega } |\nabla u |^ 2 \, dy \leq A (u,u), 
\qquad u \in \sobolev 2 1 _D ( \Omega).
\end{equation}

The conditions (\ref{Ellipticity}) and (\ref{Coerce}) 
are immediate if $D$ has non-empty interior and  
the coefficients satisfy the strong 
ellipticity condition 
$$
a^ { ij } _ {\alpha \beta} \xi _j ^ \beta 
\xi _i ^ \alpha \geq c |\xi |^ 2, \qquad \xi \in \reals ^ { 2m} .
$$
At the end of this section, we will show that they also hold for the Lam\'e
system when  the set $D$ is non-empty and satisfies the corkscrew condition. 

We are now ready to give the  precise formulation of (\ref{MP}). 
Given $f$ in $ \sobolev
2 {-1} _D ( \Omega)$ and $f_N \in \sobolev 2 { -1/2} _D ( \partial \Omega)$ 
we say that $u$ is a {\em weak solution of the mixed problem }(\ref{MP}) if we have
\begin{equation}
\label{WMP}
\left\{\begin{array}{ll}
A(u,\phi)  = - \langle f, \phi\rangle + \langle f_N, \phi \rangle _
 { \partial \Omega} , \qquad \phi \in \sobolev 2 1 _D ( \Omega ) \\
u \in \sobolev 2 1 _D ( \Omega).
\end{array}
\right. 
\end{equation}
With our continuity (\ref{Bounded}) and global coercivity (\ref{Coerce}) 
assumptions, the existence and uniqueness
of solutions is an immediate consequence of the Lax-Milgram theorem. 

We also consider the adjoint problem for the operator $L^*$ whose
coefficients are  obtained by replacing
 $ a^ {ij}_ {\alpha \beta}$ by $a^ {ji}_{\beta\alpha}$. 
We say that $w$ is a weak solution of the mixed problem for $L^*$ 
$$
\left\{ \begin{array}{ll}
L^* w = g , \qquad &\mbox{in } \Omega \\
\frac { \partial w }{ \partial \nu } = g_N, \qquad & \mbox{on } N\\
w = 0 , \qquad & \mbox{on } D
\end{array}
\right. 
$$
if we have
\begin{equation}
\label{AWMP}
\left\{\begin{array}{ll}
A(\phi, w)  = - \langle g, \phi\rangle + \langle g_N, \phi \rangle _
 { \partial \Omega} , \qquad \phi \in \sobolev 2 1 _D ( \Omega ) \\
w \in \sobolev 2 1 _D ( \Omega).
\end{array}
\right. 
\end{equation}

Given a locally integrable function on 
$ \Omega$, $ x \in \bar \Omega$,  a subset $D \subset
 \partial \Omega$, and $ \rho \in ( 0, r_0)$,   we define 
$$
  \avg u x \rho = \left \{ \begin{array}{ll}
 0, \qquad &  \mbox{if } 
 \dist(\locdom x \rho  ,D) = 0  \\
\displaystyle{  \average _ { \locdom x \rho } u(y) \, dy } , \qquad & \mbox{if } 
\dist(\locdom x \rho , D) >0
\end{array} 
\right. 
$$ 
Here, we are using $ \average _E f\, dy = 
|E|^ { -1} \int _E f \, dy $ to denote the average of $E$. We 
will define the space $BMO_D( \Omega)$ to be the collection 
of integrable functions $u$ on $\Omega$ for which the norm 
$$
\| u \|_ { *,D} =  
\sup \{ \average _ {\locdom x \rho } | u - \avg u x \rho | \, dy: x\in \bar \Omega, 
\rho \in (0, r_0) \} 
$$
is finite. The supremum is taken over all local domains. 
We will be primarily interested in the 
case when $ D \neq \emptyset$ and then it is 
easy to see that the above expression is a norm. 
In the case where $ \Omega$ is $ \reals ^2$ or 
$D = \emptyset$, then the elements of the space will be equivalence 
classes of functions which 
differ by a constant. 

We note that if $ D _1 \subset D_2$, then we have that $ BMO _ {D_1} 
( \Omega) \supset BMO _ { D_2} ( \Omega)$ and 
in particular we have $ BMO _ D ( \Omega)
 \subset BMO _ \emptyset( \Omega)$. Thus, we obtain the John-Nirenberg inequality 
 $$
 |\{ y\in \locdom x \rho : |u(y) - \avg u x \rho |  > \lambda \} |
  \leq C \rho ^ 2 \exp ( - \lambda / \| u \|_{*, D})\}.
 $$
The standard proof for $ BMO(\reals^n)$  as found in the lecture notes of Journ\'e 
\cite[Chapter 3]{MR706075}, for example, extends easily to the space 
$BMO_D( \Omega)$.  As a consequence, it follows that if $ D \neq \emptyset$, then 
$ BMO_D ( \Omega) \subset L^p( \Omega)$, for $ 1\leq p < \infty$. In the case
 where $ \Omega = \reals ^2$, we can show that any representative of a 
$ BMO( \reals^2) $ function lies in $\cap_ { p < \infty}  L^p_ { loc} 
( \reals ^2)$ and in the
case of a bounded domain any representative of a function in $ BMO_ \emptyset ( \Omega)$  also lies in $ \cap _ { p < \infty } L^ p ( \Omega)$. 

We define  atoms for $ \Omega$ and $D$ and 
then the Hardy space $H^1 _ D( \Omega)$.
We say that a bounded measurable function $a$ is 
an{ \em atom for $ \Omega$ and $D$  }if  $a$ 
is supported in 
one of the local domains $ \locdom x \rho $ and satisfies
\begin{eqnarray*}
\|a\|_ { L^ \infty ( \Omega) } & \leq & 1 /|\locdom x \rho| \\
\avg a x \rho  & = & 0 .
\end{eqnarray*}
A function $f$ is in the atomic Hardy space $H^1 _D ( \Omega)$ 
if there is a sequence of 
atoms $ \{ a_i \} _ { i =1 }^ \infty$ and a sequence of real numbers 
$ \{ \lambda _ i\} \in
 \ell ^1 $ so that $ f = \sum _ { i =1 } ^ \infty \lambda _i a _ i $. We define a 
 norm on this space 
by 
$$
\| f\| _ { H^ 1_D ( \Omega) } = \inf \sum_{i=1}^ \infty |\lambda _i | 
$$
where the infinum is taken over all representations of $f$.

It will be useful to observe that the expression 
\begin{equation}
\label{NewNorm}
\sup \{ \int _ \Omega a^ \alpha u ^ \alpha \, dy : 
a \mbox { is an atom for $\Omega$ and $D$}\}
\end{equation}
gives an equivalent norm on $BMO_D( \Omega)$. This 
  proposition may be found in Journ\'e \cite[Chapter 3]{MR706075}
for $ BMO( \reals^n)$.
 The extension to  $BMO_D( \Omega)$ is straightforward. 
We will use the characterization of the $BMO_D(\Omega) $ norm  in (\ref{NewNorm}) 
to show that the Green 
function lies in $ BMO _D( \Omega)$. 
In fact, $ BMO_D ( \Omega)$ may be identified with the dual of the
 atomic Hardy space $H^ 1 _D ( \Omega)$.  See Journ\'e \cite{MR706075}, Coifman and 
 Weiss \cite{CW:1976} for the case when $D = \emptyset$. Chang \cite{MR1253178} treats
  the extreme cases where $D = \emptyset$ or $D = \partial \Omega$. The extension to 
 general $D$ is not difficult. 
\note{ 
Show that for $p>1$, $L^p ( \Omega) \subset H^ 1 _D ( \Omega)$ when $D$
is not empty. --see if this is used?

Alternate results for $ D = \emptyset$ and $ \Omega = \reals ^2$. 

Add statement about embedding of BMO into $L^p$.--Done

Add that this result does not require that $D$ satisfy corkscrew. 

} 

We recall several Poincar\'e and Sobolev inequalities that will be needed in the 
argument below.
 The first inequality 
 is a scale invariant Sobolev-Poincar\'e 
 inequality.  If $ u \in \sobolev p 1 _ D
( \Omega)$, $ 1 \leq p < 2 $ and $1/q= 1/p-1/2$, and $ \locdom x \rho$ 
is one of the local domains defined above,  then we may find a constant 
$C$ depending only on $M$ and $p$ so that 
\begin{equation}
\label{SoPo1}
\left ( \int _ { \locdom x \rho } |u - \avg u x \rho| ^ q \, dy \right) ^ { 1/q}
\leq C \left ( \int _ { \locdom x {2\rho} }|\nabla u |^ p\, dy \right ) ^ { 1/p}
.
\end{equation}
If $ \avg u x \rho = \average _{\locdom x \rho } u \, dy $, 
the inequality holds with $ \locdom x \rho$ as the domain of integration on the 
right-hand side. 
However, if $\dist ( \locdom x \rho , D ) =0$, then we need
to expand $ \locdom x \rho$ and use the corkscrew condition  (\ref{Corkscrew}) 
in order
to conclude that $u$ vanishes on a large enough set to obtain the 
 inequality (\ref{SoPo1}). See \cite[Section 3]{OB:2009} for more details. 
A useful consequence of the inequality (\ref{SoPo1}) is 
the Poincar\'e inequality  for $ 1 \leq p < \infty $,
\begin{equation} 
\label{Poincare}
\left ( \int _ { \locdom x \rho } |u - \avg u x \rho| ^ p \, dy \right) ^ { 1/p}
\leq C\rho\left(\int _ { \locdom x {2\rho} }|\nabla u |^ p\, dy \right ) ^ { 1/p}
\end{equation}
Note that the Poincar\'e inequality (\ref{Poincare}) with $p=2$ 
 immediately implies the embedding $\sobolev 2 1 _D ( 
\Omega) \subset BMO_D( \Omega)$. The corkscrew condition is needed to establish
this embedding. 
If $ 2 < p < \infty$, we have 
the following version of Morrey's inequality
\begin{equation}
\label{Morrey}
\|u - \avg u x \rho \|_ { L^ \infty ( \locdom x \rho ) } 
\leq C\rho ^ { 1 - 2/p} \left ( \int _ { \locdom x {2\rho}} 
|\nabla u |^ p\, dy  \right ) ^ { 1/p} .
\end{equation}
This inequality gives the embedding of the Sobolev space 
$\sobolev t 1 _D( \Omega)$ into  space of H\"older continuous functions 
with exponent $ 1-2/t$ 
when $ t > 2$. 
Finally, we give an estimate at the boundary. 
Let $ \locdom x \rho$ be one of our local domains and suppose that 
$ 1 \leq p < \infty$, $1\leq q < 2$ and 
$ 1/q = 1/(2p)+ 1/2 $. There exists a constant $ C = C(p, M)$ 
so that for $u \in \sobolev q 1 _ D ( \locdom x {2\rho})$, we have
\begin{equation} 
\label{BoPo} 
\left( \int _ { \sball x \rho } |u - \avg u x \rho |^ p \, dy \right)^ { 1/p} 
\leq C \left ( \int _ { \locdom x { 2 \rho } } |\nabla u |^ q \, dy \right)
 ^ { 1/q} .
\end{equation}

To end this section, we recall the Lam\'e operator and show that the form for 
this operator with the mixed boundary condition 
satisfies the ellipticity condition (\ref{Ellipticity}) and 
coercivity condition (\ref{Coerce}). 
 Given two real-valued functions $ \mu$ and $\lambda$, the Lam\'e operator 
 is the operator with  coefficients
$$
a^ { ij } _ { \alpha \beta} (x) = \mu (x)( \delta _{ ij} \delta _ { \alpha \beta} 
+ \delta _ { j \alpha } \delta _ { i \beta }) +
 \lambda ( x) \delta_ { i \alpha } \delta _{ j \beta} 
$$
We will use $ \epsilon (u) $ to denote the {\em strain }or the 
symmetric part of the
gradient,
$$
\epsilon _\alpha ^i ( u ) 
= \frac 1 2 \left ( \frac { \partial u ^ i } {\partial x _ \alpha } + 
\frac { \partial u ^ \alpha }{ \partial x_i} \right).
$$
and then $ \sigma(u)$ will denote the {\em stress tensor }which is given by 
$$
\sigma^ i _ \alpha(u) = a ^ { ij }_{\alpha \beta } \frac { \partial u ^ \beta } 
{ \partial x _j } 
$$
or more compactly by 
$$
\sigma (u) = 2 \, \mu \,  \epsilon (u) + \lambda \, I \,  \div u 
$$
where $I$ is the $2\times2$ identity matrix. 
The functions $ \mu$ and $ \lambda$ are called the 
Lam\'e parameters and are  related to the 
elastic properties of the material. We require that $\lambda$ and $ \mu$ are 
bounded, measurable functions and that 
 $ \mu (x) - \lambda ^ - (x) \geq c> 0$ where $ \lambda ^ -$ 
denotes the negative part of the function $ \lambda$.  
With this assumption, we have the 
pointwise lower bound
\begin{equation}\label{Pointwise}
\sigma ^ i _ \alpha \frac { \partial u ^ \alpha } { \partial x_i } 
\geq c |\epsilon (u) |^2.
\end{equation}

We recall the following version of Korn's second inequality. 
Let $ \locdom x \rho  $ be one of  our local domains with star-center $x^*$.  
If $u$ is in $ \sobolev 2 1 ( \locdom x \rho)$ and $c  $ is  any
 constant vector in $ \reals ^2$, then 
\begin{equation} \label{Korn2}
\int _ {\locdom x \rho  } |\nabla u |^ 2 \, dy \leq C( \int _ { \locdom x \rho}
|\epsilon ( u) |^ 2 \, d y  +  \frac 1 { \rho ^2} \int_{ \ball {x^*}
  {\rho/2} } |u -c|^2\,dy   ) 
\end{equation}
The constant depends only on $M$.  This may be
established using the argument given in the monograph of
Ole\u\i nik, Shamaev, and Yosifian \cite[Theorem 2.10]{MR1195131}. 
It is clear that the ellipticity condition (\ref{Ellipticity}) 
follows from the pointwise
bound (\ref{Pointwise}) and (\ref{Korn2}). If $D$ is non-empty (and then the corkscrew 
condition implies $D$ has non-empty interior), the
coercivity condition (\ref{Coerce}) may also be 
found in Ole\u\i nik, Shamaev, and Yosifian 
\cite[Theorem 2.7]{MR1195131}, 
however the standard proof of this inequality seems to be by contradiction and thus
we cannot say anything about how the constant depends on the domain $ \Omega$ and the
boundary set $D$. 

All of our quantitative assumptions have been in terms of the constant $M$. 
The results 
below will be  of  two types. Many of the results will be local estimates 
which hold on 
scales $\rho \in ( 0, r_0)$. In these local estimates, 
the constant will depend only on $m$, $M$, 
and any $L^p$-indices that appear in the estimate. The remaining estimates will
depend on global properties of the domains such as the collection of coordinate
cylinders which cover the boundary or the constant in the
Poincar\'e inequality  (\ref{DMatters}). 
However, the constant may be chosen to be uniform 
 under small changes in the Lipschitz functions
which define the boundary. One exception is the constant in the  Korn
inequality which gives the coercivity condition (\ref{Coerce}) for the
Lam\'e system.  As noted above, 
the  standard proof is by contradiction and gives no
information about the dependence of the constant on the domain. 
In the study of the Neumann problem the global estimates for solutions
will depend on the estimate in our existence theorem, Theorem \ref{SolveN}. The proof
of this result uses the Fredholm theory and thus we have no information about
the behavior of the operators used to solve the Neumann problem. 

\section{Estimates for solutions of the mixed problem}
\label{Estimates}
In this section, we prove  estimates for solutions of the mixed problem.
 We will use the 
reverse H\"older technique of Giaquinta and Modica \cite {MR549962}. 
Before beginning the 
main argument, we introduce
two auxiliary functions which will arise when we prove Caccioppoli-type
 inequalities for 
solutions of the mixed problem (\ref{WMP}) when the data $f$ or $f_N$ is not zero.

In the following definitions, we assume that $ 0 < r< r_0$ so that we can
define $ \locdom x \rho$ and $ \sball x \rho$  for $ 0 <  \rho < r$. 
For $ 0 < \alpha < 2$,  we let $I _{ \alpha , r} f (x) $ denote the 
maximal fractional integral given by 
$$
I_{\alpha,r }  f(x) 
= \sup _ { 0 < \rho < r} \rho ^ { \alpha -2 } 
\int _ { \locdom x \rho } |f(y) | \, dy .
$$
Note that we have several choices for $ \locdom x \rho$ when $x$ lies in several 
coordinate rectangles. In defining $I_{ \alpha, r}$, we take the maximum 
 over the choices  for $ \locdom x \rho$ arising from 
our finite cover of the boundary by coordinate rectangles.  
We have that 
$ \rho ^ { \alpha -2 } \chi _ { \locdom x \rho } ( y ) 
\leq C(M) |x- y | ^ { \alpha -2}$ and it follows that 
$ I _{\alpha, r} f(x) \leq  C R_ \alpha |f| (x)$ where $ R_ \alpha$ is the 
standard Riesz potential. Thus, the Hardy-Littlewood-Sobolev theorem 
gives us that with $q$ and $p$ related by $1/q=1/p+\alpha/2$ and  $ 2/ ( 2- \alpha ) 
< p < \infty  $, 
we have
\begin{equation}
\label{MaxFrac}
\| I_ {\alpha , r}  f \| _ { L^ p ( \Omega ) } 
\leq C(p, M) \| f \| _ { L^ q ( \Omega) } .
\end{equation}
Next, if $ f_N$ is a function on $ \partial \Omega$, we define
$$
P_ rf_N(x) =
 \sup \{ \average _ {  \sball { \hat x } \rho   } |f_N|\, d\sigma :
 \dist(x, \partial \Omega)<  \rho < r \} 
$$
with the convention that the supremum of the empty set is zero. 
Recall that if $ x $ is in a coordinate rectangle $\cyl y r$, 
and $ x = ( x_ 1 , \phi(x_1) + te_2)$, we defined $ \hat x = (x_1, \phi(x_1))$. 
We also note that if $ \dist (x, \partial \Omega) \leq \rho$, then $ \Delta _ \rho ( \hat x)
= \partial \Omega \cap  \partial \locdom x \rho $ and if $ \dist ( x,
\partial \Omega ) > \rho$, 
then $\partial \Omega \cap  \partial \locdom x \rho =\emptyset$. 
Our Neumann data, $f_N$ is initially only defined on $N$. 
In the definition of $P_r$, we assume that $f_N$ has been extended 
to $ \partial \Omega$ by setting $f_N=0$ outside $N$. Following the argument 
in \cite[Section 3]{OB:2009}, for $ 1 < p  < \infty$, we may find a 
constant $ C= C(M,p, \Omega)$ 
so that we have the estimate
\begin{equation}
\label{Pmax}
\left( \int _ \Omega | P_rf | ^ { 2p } \, dy \right) ^ { 1/2p }
\leq C \left ( \int _ { \partial \Omega } |f|^ p \, d\sigma \right ) ^ { 1/p } 
\end{equation}
and the corresponding estimate for $ p = \infty$ is trivial. 
\note{ Should we have a local version of this estimate, with a constant depending 
only on $M$ and $p$. }

We are now ready to begin our estimates. We begin with a simple energy estimate. 
\begin{proposition}
\label{EnergyProp}
If $ f$ lies in the Hardy space $H^ 1_D( \Omega)$ and $ f_N = 0$, then 
the weak mixed problem (\ref{WMP}) has a unique solution $u$ which satisfies
$$
\int_ \Omega |\nabla u |^ 2\, dy \leq   C\| f\| ^ 2 _ { H^1 _D( \Omega)}.
$$
\end{proposition}

\begin{proof} Let $a$ be an atom for the Hardy space $H^1_D( \Omega)$. 
Using the Poincar\'e inequality (\ref{Poincare}), it is easy to see that the
map $ u \rightarrow \int _{ \Omega } a^ \alpha u^ \alpha\, dy $ lies in the
dual of $ \sobolev 2 1 _D ( \Omega)$. As a consequence, the solution of the 
mixed problem (\ref{WMP}) with $f=a$ and $f_N=0$ will satisfy 
$\int _  \Omega |\nabla u |^2 \, dy \leq C$. The result for a general element
of $H^1_D( \Omega)$ follows easily from the result for an atom. 
\end{proof}

  The following theorem gives our main
estimate for solutions of the mixed  boundary value problem. We note that the 
estimates of this theorem hold for the extreme cases where $D = \emptyset$ or 
$\partial \Omega$.   In the arguments below, we will use the standard notation
$1/p' = 1 -1 /p$ to denote the conjugate exponent. 
\begin{theorem}
\label{CaccTheorem}
Let $ \Omega$ be a Lipschitz domain, let $D \subset \partial \Omega$ be a
(possibly  empty) 
set that satisfies the corkscrew condition (\ref{Corkscrew}), and 
suppose that $L$ 
is an elliptic operator which satisfies the conditions (\ref{CoeffBound}) and 
 (\ref{Ellipticity}). 
 Let $u$ be a weak solution of the mixed 
problem (\ref{WMP}) and assume that
$f \in L^p ( \Omega)$ and $f_N \in L^ { p_N} ( N) $ 
are functions.  We  may find a $ t_0 >2$ 
so that with $ t \in [2,t_0)$,  $ p \in (1,2)$ and $p_N \in (1,\infty)$, $x 
\in \bar \Omega$ and $ \rho \in (0, r_0)$ we have the estimate
\begin{eqnarray*}
\left( \average _ { \locdom x \rho } |\nabla u |^ t \, dy \right ) ^ { 1/t} 
&\leq& C \left [ \left ( \average _{ \locdom x { 2\rho } }
 |\nabla u |^2 \, dy \right ) ^ { 1/2} 
+ \left( \average _ { \locdom x { 2\rho } } 
I _ {p, {2\rho}} (|f|^ { p } ) ^ { t/p } \, dy \right) ^ { 1/t } \right. \\
& &\left. \qquad+ \left(\average_{ \locdom x { 2\rho} } 
P_{2\rho}(|f_N| ^ { p_N} ) ^ {t/p_N}\,dy \right) ^ { 1/t}\right].
\end{eqnarray*}
The exponent $t_0 $ depends on $M$, $p$, and $p_N$ and the constant $C = C(M,t,p,p_N)$. 
\end{theorem}

\begin{proof}
We let $u$ be a weak solution of the mixed problem (\ref{WMP}) 
with $f\in L^ p ( \Omega) $ and $f_N \in L^ { p_N}(N)$ 
 for some  $ p>1$ and $p_N >1$. Define $q$ by $ 1/q = 3/2-1/p$ 
and $q_N$ by $ 1/q_N = 1- 1/( 2p_N)$ and let $ s= \max (q, q_N) < 2$. 
We fix $ x\in \bar \Omega$, $ 0 < \rho < r _0$ and 
we claim the Caccioppoli inequality
\begin{multline}
\label{Cacc0}
\average _ { \locdom x \rho } |\nabla u |^2 \, dy 
\\
\leq C \left [ \left(\average_ {\locdom x { 4\rho} } 
|\nabla u |^ s \, dy \right) ^ { 2/s}
+ I_{p, 2\rho} (|f|^ p ) ^ { 2/p }(x)
+ P _{2\rho}(|f_N|^ { p_N} ) ^ { 2/ p _N}(x)\right]
\end{multline} 
The estimate of the theorem follows immediately from (\ref{Cacc0}) and
the reverse H\"older argument of Giaquinta and Modica (see 
\cite[Theorem 1.2, p.~122]{MG:1983}). To establish (\ref{Cacc0}),
we fix $ x\in \bar \Omega$ and $ \rho \in (0, r_0)$. We let $ \eta $ 
be a cutoff function which
is one on $ \locdom x \rho$,  zero on $ \Omega \setminus \locdom x { 2\rho }$,
and satisfies  $ |\nabla \eta | \leq C /\rho$. We let 
$E = \int _ { \locdom x {2\rho } } |u-\avg u x {2\rho} | \eta 
|\nabla \eta | | \nabla u | 
+ \rho^{- 2} |u - \avg u x {2\rho} |^ 2 \, dy .$
As our first step, we will establish the inequality
\begin{multline}
\label{Claim1}
\frac 1 {\rho ^ 2} \int _ { \locdom x { 2 \rho }}
\eta ^2|\nabla  u | ^ 2\, dy 
\leq 
C\left (  \frac 1 {\rho^2} E 
+   \left( \average_ { \locdom x { 4 \rho}} |\nabla u |^ s \, dy \right ) ^ { 2/s}
\right.
\\
\left. 
+ \left ( \frac 1 { \rho ^ { 2-p }} \int _ { \locdom x { 2\rho }}|f|^ p\,dy
\right) ^ { 2/p} 
+ \left( \average _{\partial \Omega \cap \locdom x {2\rho} } |f_N |^ { p_N} \, d\sigma 
\right) ^ { 2/p_N}\right).
\end{multline}
If we use (\ref{SoPo1}) and the Cauchy inequality with $ \epsilon$'s,   
the estimate  (\ref{Claim1}) implies (\ref{Cacc0}). 
Thus we turn to the proof of 
(\ref{Claim1}).

To establish (\ref{Claim1}), observe that thanks to our definition of 
$\avg u x \rho$, we have  $ v = \eta ^2 ( u - \avg u x {2 \rho } )
\in \sobolev 2 1 _D ( \Omega)$ whenever $u \in \sobolev 2 1 _ D ( \Omega)$.
Thus, from the product rule, our ellipticity assumption 
and the weak formulation of
the mixed problem, (\ref{WMP}), we obtain 
\begin{equation}
\label{One}
\begin{split}
\int_ \Omega \eta^ 2 |\nabla u |^ 2 \, dy 
&\leq  \int _ { \Omega } |\nabla ( \eta ( u - \avg u x {2 \rho} ) )|^2\, dy+CE
\\
& \leq  C(A(\eta( u - \avg u x {2\rho} ) , \eta (u - \avg u x {2\rho} ) ) +E)
\\
 &\leq  C( A(u,v) + E) 
\\
& \leq C( \int_{\partial \Omega} f^\alpha _N v^ \alpha \, d\sigma  - 
 \int_{ \Omega } f^ \alpha v^ \alpha \, dy + E)
\end{split}
\end{equation}

\note{ If $f \in L^p$, $p\geq 2$, 
   then we also have $f$ in $L^p $, $p<2$,  and the result applies. }
We claim that if $ 1 < p  <  2$,  we have 
\begin{equation}
\label{Three} 
\left | \frac 1 { \rho ^ 2} \int _ \Omega f ^ \alpha v ^ \alpha \, dy \right|
\leq  \left ( \average _ { \locdom x { 2\rho }}
 |\nabla u | ^ q \, dy \right ) ^ { 2/q} +
C\left ( \frac 1 { \rho ^ { 2-p } } 
\int _ { \locdom x { 2\rho } } |f |^ p \, dy\right) ^ { 2/p}.
\end{equation}
We also will need that if $ 1 < p _N \leq \infty $,  then
\begin{multline}
\label{Four}
\left | \frac 1 { \rho ^ 2} \int _N f_N ^ \alpha v^ \alpha \, d\sigma \right |
\\
\leq   \left ( \average_ { \locdom x { 2\rho}} 
|\nabla u | ^{q _N}\, dy \right) ^ { 2/q_N}
+ C\left( \average  _ { \partial \Omega \cap \partial \Omega _ { 2\rho } (x)}
 |f_N |^ { p_N} \, d\sigma \right) ^ 
{2/p_N }  .
\end{multline}
It is easy to see that (\ref{Claim1}) follows from (\ref{One}-\ref{Four}) and  
thus it remains to prove (\ref{Three}) and (\ref{Four}).

To establish (\ref{Three}), we use H\"older's inequality with $ 1< p <
2$   and then 
the Sobolev-Poincar\'e inequality (\ref{SoPo1}) to obtain that with
$1/q =3/2-1/p$ 
\begin{eqnarray*}
\left | \frac 1 {\rho^2} \int _ { \Omega } f^ \alpha v^ \alpha\, dy \right |
&\leq & \frac 1 { \rho ^2} \left( \int _ { \locdom x { 2\rho } } |f |^ { p }
 \,dy
\right) ^ {1/p}
 \left( \int _ { \locdom x { 2\rho }}|u - \avg u x { 2 \rho } | ^ {p'} 
 \, dy \right ) ^ { 1/ p'} \\
&\leq  & \frac C { \rho ^{2/q -1 + 2/p}} 
\left( \int _ { \locdom x { 2\rho } }|f |^ p \, dy \right) ^ { 1/p } 
\left( \int _ { \locdom x { 4\rho} }|\nabla u |^ q \, dy \right) ^ { 1/q} .
\end{eqnarray*}
Now (\ref{Three}) follows easily.  

To establish (\ref{Four}), we use H\"older's inequality and the boundary Poincar\'e
inequality (\ref{BoPo}) to obtain with $ 1 < p_N <\infty $ and 
$ 1/q_N = 1- 1 / ( 2p_N)$, that 
\begin{eqnarray*}
\lefteqn {\left | \frac 1 { \rho^2} \int _ N f_N^ \alpha v ^ \alpha \, d\sigma \right | } 
\\
& \leq & \frac 1 { \rho^ 2} 
\left ( \int _ { \partial \Omega \cap \partial \locdom x { 2\rho }  } 
|f_N |^ { p_N} \, d\sigma \right) ^ {  1 /{p_N}}
  \left (\int _  {\partial \Omega \cap \partial \locdom x { 2\rho} } 
  |u- \avg u x { 2\rho }| ^ { p'_N} \, dy \right)^ {  1 /{p'_N}} \\
&\leq & \frac C { \rho ^ { 1/p_N  + 2 / q_N} } 
\left ( \int _{ \partial \Omega \cap \partial \locdom x { 2\rho }}
 |f_N | ^ {p _N} \, d\sigma \right ) ^ { 1/ p_N } 
\left( \int _ {\locdom x { 4 \rho }  } |\nabla u |^ { q_N} \, dy\right )
 ^ {  1/ { q _N} }  
\end{eqnarray*}
and now (\ref{Four}) follows. 
\note{ If $p_N = \infty$, then the argument still holds. The only reason to exclude 
$p_N =\infty$ is to avoid having to write a separate line to acccount for the fact
that the $L^ \infty$ norm is not an integral.}
\end{proof}

\begin{remark} 
The argument above also gives us a Caccioppoli inequality for solutions.
Let $u$ be a solution of (\ref{WMP}) with $ f =0$ and $ f_N \in L^ 2(N)$, fix
$ x $ and $ \rho$ and let $v$ be as in (\ref{One}).
Since $v$ vanishes outside 
$ \locdom x {2\rho}$, we may use (\ref{BoPo}) and H\"older's inequality to
obtain
$
\int _{\partial  \Omega } |v|^2 \, d\sigma \leq C \rho \int _{ \Omega }
 |\nabla v |^ 2 \, dy .
$
If $ \dist(x, \partial \Omega ) > 2\rho$, the boundary term  in (\ref{One}) 
is zero.  When
$\dist(x, \partial \Omega ) < 2 \rho$,   we estimate the 
 boundary term in (\ref{One}) using the Cauchy
inequality with $\epsilon$'s and the above observation to obtain 
$$
|\int _N f_N^ \alpha v^ \alpha \, d\sigma|
 \leq C(\frac \rho \epsilon \int _{ \sball x { 2\rho } } |f_N |^2 \, d\sigma
 +E) 
 + \epsilon \int _ \Omega \eta ^ 4|\nabla u |^2 \, dy  .
 $$
Using this estimate in (\ref{One}) and the Cauchy inequality with $ \epsilon$'s 
gives 
\begin{equation}
\label{Caccioppoli}
\int _ { \locdom x \rho } |\nabla u |^ 2 \, dy \leq C(  \frac 1 { \rho ^ 2}
\int  _ { \locdom x { 2\rho }}
|u - \avg u x {2\rho} | ^ 2 \, dy  + 
\rho\int _ {\partial \Omega \cap \partial \locdom x  { 2 \rho} } f_N^2 \, d\sigma).
\end{equation}
\end{remark}

\note{ 
We give a more detailed proof.

We first note that a variant of  the boundary sobolev inequality  (\ref{BoPo}) 
and then an application of H\"older's inequality 
gives
$$
\int_ { \partial \Omega } |v|^2 \, d\sigma  \leq C\left (
 \int _ { \Omega } |\nabla v|^{ 4/3}
\, dy \right) ^ { 3/2} \leq C \rho \int _ \Omega |\nabla v|^2 \, d\sigma 
$$

To prove this we use the product rule and ellipticity to obtain
\begin{eqnarray*}
\int \eta^ 2 |\nabla u |^ 2 \, dy 
& \leq & C (\int _ { \Omega } |\nabla ( \eta ( u - \avg u x \rho ) )|^2\, dy+E)
\\
& \leq & C(A(\eta( u - \avg u x \rho ) , \eta (u - \avg u x\rho ) ) +E)
\\
& \leq & C( A(u,v) + E) 
\\
& \leq &C( \int f^\alpha _N v^ \alpha \, d\sigma + E)
\\
& \leq & C( ( \int _ { \sball x { 2\rho }} |f_N|^2 \, d\sigma ) ^ { 1/2} 
(\int _ { \sball x { 2\rho } } |v|^ 2 \, d\sigma ) ^ { 1/2} +E)  \\
& \leq & 
C(( \int _ { \sball x { 2\rho }} |f_N|^2 \, d\sigma ) ^ { 1/2} 
 \rho ^ { 1/2} ( \int _ \Omega |\nabla v |^ 2\, dy ) ^ { 1/2 } +E) \\
 & \leq & 
C(\frac \rho \epsilon \int _ { \sball x { 2\rho } } |f_N|^2 \, d\sigma
 + \epsilon \int _ \Omega  \eta ^4|\nabla u |^ 2\, dy 
 +E) 
\end{eqnarray*} 
We will use the Caccioppoli inequality 
with $f_N$ nonzero when we construct the Green
function for the Neumann problem. 

The proof of CaccTheorem may allow larger values of $p_N$. Can we let $p_N < t$? 
}
\begin{corollary}
\label{HolderCor}
Let $ \Omega$,  $D$, and $L$ be as in Theorem \ref{CaccTheorem}. 
Let $u$ be a  solution of the weak mixed problem (\ref{WMP}) with
$f=0$ on $\locdom x { 2\rho } $ and $f_N \in L^2 ( \partial \Omega)$. 
There is a constant $C$ and $ \gamma_0 >0$ so that for
$ \gamma$ in $ ( 0, \gamma _0)$, we have 
\begin{gather}
\begin{split} 
\label{Holder}
|u(y) - u ( z) | \leq C \frac {|y-z|^ \gamma}{ \rho ^ \gamma}
( \left (\average_{ \locdom x { 2\rho } } |u - \avg u x {2\rho } |^ 2\, dy
\right)^{1/2}  \\ 
 + \left ( \rho \int _ { \partial \Omega \cap \partial \locdom x { 2\rho} } |f_N |^ 2 \,d\sigma \right) ^
 { 1/2 } ) 
, \  y, z \in \locdom x \rho
\end{split} 
 \\
\label{uBound}
|u ( x ) | \leq   C (  \left( \average _ { \locdom x {2\rho } }
 |u |^ 2 \, dy \right) ^ { 1/2}
 +\left ( \rho \int _ { \partial \Omega \cap \partial \locdom x { 2\rho} } 
 |f_N |^ 2 \,d\sigma \right) ^
 { 1/2 })
\end{gather}
The constant  $ \gamma_0$ depends only on $ M$ and $C$ depends 
on $M$ and $ \gamma$.  
\end{corollary}

\begin{proof} 
We let $u$ be a solution of (\ref{WMP}) with $f=0$ in $\locdom x {2\rho}$.  
We observe that there is a constant $c$ so that if 
$y \in \locdom x {3\rho/2}$, $0 < s< c\rho $, then 
$ \locdom y { 2s} \subset \locdom x { 2\rho }$ and we have 
$ I_ { \alpha, 2s} f =0$.
Thus if $ \locdom y s \subset \locdom x \rho$ and $ s < c\rho$, 
the main estimate of Theorem \ref{CaccTheorem}  and (\ref{MaxFrac})
imply   that for $ t < \min (4, t_0)$
\begin{equation} 
\label{RH} 
\left( \average _ { \locdom y s } |\nabla u |^ t \, dy \right) ^ { 1/t}
\leq C( \left( \average _ { \locdom y {2s} } |\nabla u | ^ 2 \, dy \right)^ { 1/2}
+ \frac \rho s \left( \average _ { \partial \Omega \cap \partial \locdom x { 2\rho} }
|f_N |^2 \, d\sigma \right )  ^{1/2} ).
\end{equation}
Using (\ref{RH}) with the Caccioppoli inequality (\ref{Caccioppoli}) 
and Morrey's inequality 
(\ref{Morrey}), we obtain the H\"older estimate (\ref{Holder}) with $ \gamma_0
= \min(  1- 2/t_0, 1/2)$ where $t_0$ is as in Theorem
\ref{CaccTheorem}.  

To obtain the estimate (\ref{uBound}), we write 
$$
|u(x) | \leq |u(x) - \avg u x \rho | + |\avg u x \rho|, 
$$
 use the H\"older estimate (\ref{Holder}) 
to bound the first term on the right and then the H\"older inequality. 
Note that this 
works even when $ \avg u x \rho =0$ since in this case, $u$ vanishes on a nearby 
piece of the boundary.
\end{proof}

We observe a version of the Green identity for solutions of 
the weak mixed problem (\ref{WMP}) and the corresponding problem for the
adjoint operator, $L^*$. If $ u $ satisfies $ Lu =f$ with  
$ \partial u /\partial \nu = f_N$ and
$w$ is a solution of the adjoint problem $ L^* w =g$ with 
$\partial w /\partial \nu = g_N$ and suppose
that $f$ and $g$ are in $L^p( \Omega)$  and $ f_N$ and $g_N$ are
in $L^ {p}( \partial \Omega)$ for some $p>1$. Then,  we have 
\begin{equation}
\label{Dual}
\int _ \Omega u^ \alpha g^ \alpha\, dy -
\int _{N } u ^ \alpha g^ \alpha _N \, 
d\sigma   = \int _ \Omega w^ \alpha f^ \alpha\,d y
- \int _N w^ \alpha f_N ^ \alpha \, d\sigma  .
\end{equation}
The identity (\ref{Dual}) follows easily since both sides of (\ref{Dual}) are
equal to $A(u,v)$.

The following Corollary is only valid  when $D \neq \emptyset$. 
For the Neumann problem,  there are non-constant solutions
to the 
 homogeneous problem and hence the estimates (\ref{PrimalEstimate}) and (\ref{DualEstimate}) 
 cannot hold. 
We will give a version
of this result for the Neumann problem in section \ref{Neumann}. 

\begin{corollary}
\label{TwoEstimates}
Suppose that $D$ is a non-empty 
subset of $ \partial \Omega$  that satisfies
the corkscrew condition.  
Let   $L$,  $ \Omega$,  and $t_0$   be as in Theorem \ref{CaccTheorem}, 
suppose that $t \in (2, t_0)$, define $r$ by 
 $ 1/r = 1/2 + 1/t$ and let $ r_N = t/2$. 
 If $u$ is a weak solution of the mixed problem 
(\ref{WMP})
 with 
$f_N  =0$ and $f \in L^ r ( \Omega)$, then we have 
\begin{equation}
\label{PrimalEstimate}
\| \nabla u \| _ { \lp t ( \Omega )} \leq C  \| f\|_ { \lp r( \Omega) } . 
\end{equation}

If $u$ is a solution of (\ref{WMP}) with 
$f\in \sobolev {t'} { -1} _ D ( \Omega)$ and $ f_N=0$, then
we have  the estimate
\begin{equation}
\label{DualEstimate}
\|u \|_ { L^ { r'} ( \Omega) } \leq 
C\|f\| _ { \sobolev  { t' } { -1} _D ( \Omega)}.
\end{equation}
The constant $C$ depends on $ M$, $t$,  and $ \Omega$. 
\end{corollary}

\begin{proof}
Using the estimate (\ref{DMatters}), it follows that the 
map 
$ v \rightarrow  \lambda (v) = \int _ \Omega f^ \alpha v^ \alpha\, dy $ is 
an element in  $ \sobolev 2 { -1} _D ( \Omega)$ with the bound $
\|\lambda \|_ { \sobolev 2 { -1} _D ( \Omega) } 
\leq C r_0 ^ { 2/r'}  \| f \|_ { L^ r ( \Omega)}  $.
As a consequence, if $u$ is a weak solution of the mixed problem
(\ref{WMP}) with $f \in L^ r ( \Omega)$, $ r > 1 $ and $ f_N =0$, 
 we have the energy estimate
\begin{equation}
\label{Energy}
\|\nabla u \| _ { L^ 2 ( \Omega) } \leq C 
r_0 ^ { 2/r'} \| f \|_ { L^ r ( \Omega)} . 
\end{equation}

We begin with the main estimate of Theorem \ref{CaccTheorem} on domains 
$\locdom x  { r_0} $.   We choose $ p$ with $ p \in (1,r) $ 
 and apply the estimate  (\ref{MaxFrac})  
to obtain
$$
\left( \int _ { \locdom x {r_0}  } |\nabla u |^ t \, dy \right)^ { 1/t}
\leq C ( r_ 0 ^ { 2/t-1}  \left ( \int_ { \locdom x { 2r_0}} | \nabla u | ^ 2\, dy
\right) ^ { 1/2 }  + \| f \|_{ L^ r ( \Omega) }).
$$
We use the energy estimate (\ref{Energy}) 
and observe that $1 - 2/t = 2/r' $.  Finally, 
we  cover $ \Omega$ with domains $ \locdom x {r_0}$ 
 and use Minkowski's inequality to obtain the
first estimate  (\ref{PrimalEstimate}).

The second estimate (\ref{DualEstimate}) for solutions of $L$ follows 
by duality from  (\ref{Dual}) and the first
estimate (\ref{PrimalEstimate})  for solutions of the adjoint operator $L^*$. 
If $ A^*$
is the form for $L^*$, then $ A^*(u,v) = A(v,u)$, thus it is clear that $L^*$ 
satisfies the hypotheses of this theorem exactly when $L$ does. 
\end{proof}
\note{ The proof above was edited to get rid of an aborted attempt to
include the case of non-zero Neumann data.
}

Our next estimates are valid only when $ D \neq \emptyset$. The estimates may fail 
for the Neumann problem due to non-uniqueness. 
\begin{corollary} 
\label{AtomicCorollary} Suppose that $ D\neq \emptyset$ and satisfies the corkscrew 
condition. 
If $u$ is a solution of the weak mixed problem  (\ref{WMP}) with $ f$ 
an atom and $f_N =0$, then we have
$$
\| u \| _ { L^ \infty ( \Omega ) } \leq C.
$$
\end{corollary}

\begin{proof}
We let $u$ be a solution of the weak mixed problem with $ f =a$, an
atom for the Hardy space $H^1 _D( \Omega)$, and $ f_N=0$. We suppose that
$a$ is supported in $ \locdom x \rho$. 

To estimate $u$, we fix $z$ in $ \Omega$, $t>2$, and use the Morrey inequality
(\ref{Morrey}) and the H\"older inequality  to obtain 
\begin{eqnarray}
\nonumber
|u(z)| & \leq & |u(z) - \avg u z \rho | +  |\avg u z \rho|\\
\label{StepA} 
&\leq&
 C ( \rho ^ { 1 - 2/t} \|\nabla u \|_ { L^ t ( \Omega )}
  + \rho ^ { 2/t -1} \|u \|_ {L^ { r' } ( \Omega)}).
\end{eqnarray}
As in Corollary \ref{TwoEstimates}, $r$ is defined by $ 1/r = 1/2 + 1/t$ 
and $r'$ is the dual exponent given
by $ 1/r' = 1/2 -1/t$.

We restrict $t$ to lie in the interval $(2,t_0)$ with $t_0$ as in 
Theorem \ref{CaccTheorem} and show how to use 
Corollary \ref{TwoEstimates} to estimate the two terms on the right-hand side of 
(\ref{StepA}). First, we use (\ref{PrimalEstimate}) and the normalization 
of the atom to conclude that 
\begin{equation}
\label{StepB}
\|\nabla u \|_ { L^ t ( \Omega)} \leq C \| a \|_ { L^ r ( \Omega)} 
\leq C \rho^ { 2/t -1}.
\end{equation}
Next, we claim that for $ t > 2$, we have 
\begin{equation}
\label{DualClaim}
\|a\|_ { \sobolev {t'} {-1}( \Omega) } \leq C \rho ^ { 1 - 2/t}.
\end{equation}
Given the claim, the estimate (\ref{DualEstimate}) in Corollary \ref{TwoEstimates}
implies that we have 
\begin{equation}
\label{StepC}
\| u \| _ { L^ { r' } ( \Omega)} \leq C \rho ^ { 1-2/t}.
\end{equation}
To establish the claim (\ref{DualClaim}), we use that either
 $ \avg v x \rho =0$ or $ \int _ \Omega a \, dy =0$,
the normalization of the atom, and the Morrey inequality (\ref{Morrey}) to 
write
$$
\left | \int_ \Omega  a^ \alpha v ^ \alpha \, dy  \right | =
\left | \int_ \Omega a^ \alpha ( v - \avg v x \rho ) ^ \alpha \, dy \right |
\leq C \rho ^ { 1- 2/t} \|\nabla v \|_ {L^ t( \Omega)}
$$
which gives the claim  (\ref{DualClaim}). 

The estimate of the Corollary follows from (\ref{StepA}), (\ref{StepB}), 
and (\ref{StepC}). 
\end{proof}

\section{The Green function for the mixed problem}
\label{Green}

In this section, we prove the existence of the Green function for the mixed  problem and give its main 
properties. We allow the case when $ D= \partial \Omega$ which gives the 
Dirichlet problem.  

We begin by giving 
 our definition of a Green function. This formulation is modeled on the 
definition in Littman, Stampacchia, and Weinberger 
\cite{LSW:1963}. We say that $ G(x,y) = (G^ { \alpha \beta}( x, y))_ {\alpha,\beta=1, 
\dots,m}$ is 
a {\em Green function with pole at $x$ for the 
mixed problem (\ref{WMP}) }if 
 $ G(x, \cdot ) \in L^ 1 ( \Omega)$ and whenever 
$u$ is the weak solution of the mixed problem with 
$ f \in C( \bar \Omega)$ and $f_N=0$, then we 
have 
$$
u^ \alpha (x) = \int _ \Omega G^ { \alpha \beta } (x,y ) f^ \beta (y) \, dy.
$$
If we have existence and uniqueness for the mixed problem (\ref{WMP}) 
when the right-hand side, 
$f$ is in $ C( \bar \Omega)$, then it is obvious that
the Green function is unique. The following theorem shows that the 
Green function exists 
and gives additional regularity of the Green function.
 
\begin{theorem}
\label{GreenTheorem}
Let $ \Omega$ be a bounded Lipschitz domain in the plane and let $ D \subset 
\partial \Omega$ be a non-empty open set satisfying 
 the interior corkscrew condition. If $L$ is an elliptic operator satisfying
 (\ref{CoeffBound}), (\ref{Ellipticity}), and  (\ref{Coerce}) on 
$ \sobolev 21 _D( \Omega)$.
  Then there exists a unique Green function
 $G(x,y)$ for the mixed problem which satisfies 
\begin{eqnarray}
\label{ItzaBMO}
& & G(x, \cdot ) \in BMO_D( \Omega) , \qquad x \in \Omega\\
& & 
\label{GradientBound}
\nabla _ y G(x, \cdot ) \in L^2 ( \Omega \setminus \locdom x r ) , 
\qquad\mbox{for } x\in \Omega, \  r \in ( 0, r_0)\\ 
& & 
\label{LogBound}
|G(x,y) |  \leq C( 1+ \log ( d /|x-y|)), \qquad x, y \in \Omega \\
&&
\label{GreenHolder}
 |G(x,y) - G(x,z) |  \leq  C \frac { |y-z|^ \gamma }{ |x-y | ^
\gamma } , \qquad x,y,z \in \Omega , \ |y-z|< \frac 1 2 |x-y|  \\
&&
\label{Boundary} 
|G(x,y ) | \leq C \frac { \dist(y,D)^ \gamma } {|x-y | ^ \gamma } , \qquad x, y 
\in  \Omega , \  \dist ( y, D) < \frac 1 2 |x-y | .
\end{eqnarray}
In these estimates $d$ is the diameter of $ \Omega$ and the constants depend
on the global character of $ \Omega$ as well as $M$. 

If we let $ G$ be the Green function for the mixed problem for 
$L$ and $\tilde G$ the
Green function for $L^*$ then we have 
\begin{equation}
\label{Symmetry}
G^ { \alpha\beta} (x,y) = \tilde G ^ { \beta \alpha } (y,x), \qquad \alpha,\beta=
1, \dots,m . 
\end{equation}

Furthermore, if $ f \in L^p ( \Omega)$ and $f_N \in L^ p ( N)$ for some $p> 1$, 
then the unique solution of the weak mixed problem (\ref{WMP}) is given by 
\begin{equation}
\label{Rep}
u^ \alpha (x) =  \int _ \Omega G^ { \alpha \beta } ( x,y ) f^ \beta ( y ) \, dy 
- \int _ { N } G^ { \alpha \beta } ( x, y ) f_N ^
 \beta ( y) \, d\sigma.
\end{equation}
\end{theorem}

\begin{proof} 
We fix $ x \in \Omega$, $ \alpha $ in $\{ 1, \dots, m\}$,  and
$ \rho \in (0, r_0)$. We let $ G_ \rho ^ { \alpha\cdot}  ( x, \cdot ) $ be the 
solution of mixed problem (\ref{AWMP})
for $ L^*$ with this $g = e_ \alpha \chi _ { \locdom x \rho } / |\Omega _ 
  \rho (x) |$ and $ g_N = 0$. 
We fix $f$ in $L^ p ( \Omega)$ and $f_N \in L^{p_N} (N)$ with $ p > 1 $ and $p_N>1$
and 
let $u$ be  a weak solution of the mixed problem for $L$ (\ref{WMP}) with data $f$
and $f_N$. 
If we let $ w = G_ \rho ^ { \alpha \cdot } ( x,\cdot )$ in (\ref{Dual}), we obtain
\begin{equation}
\label{RepApprox}
\average _ { \locdom x \rho } u^ \alpha (y )\, dy  
= \int _ \Omega G^ { \alpha \beta } _ \rho ( x, y ) f^ \beta (y) \, dy 
- \int _ { N } G^ { \alpha \beta } _ \rho (x,y)f ^ \beta _N (y)\, d\sigma.
\end{equation}
If  $ f_N$ is 
zero and  $f$ is an atom for $H^1_D( \Omega)$, we may use the estimate of 
Corollary \ref{AtomicCorollary} and the equivalent  norm on $BMO_D( \Omega)$ given
in (\ref{NewNorm})
to conclude that 
$ \| G_\rho ( x, \cdot ) \| _ { *,D} \leq C$ with the constant $C$ depending
only on  $M$ and the global properties of $ \Omega$.   The Banach-Alaoglu theorem
 gives
that for each $x$, there is a sequence $ \{\rho _j\}$ 
with $ \lim _ { j \rightarrow
\infty} \rho _j =0 $ and a function $ G(x, \cdot) \in BMO_D(\Omega)$ so 
that $ G_ {\rho _j }( x, \cdot )$ converges to $G(x, \cdot)$ in the weak-$*$ topology of 
$BMO_D( \Omega)$.   Since $u$ is H\"older continuous, the left-hand side of 
(\ref{RepApprox}) converges to $u^ \alpha(x)$.  
Since $ L^ p ( \Omega ) \subset H^ 1 _D ( \Omega)$ for all $p > 1 $, 
we obtain the representation (\ref{Rep}) in the case that $ f_N =0$. This gives
us that $ G(x, \cdot)$ is a 
 Green function for 
the mixed problem with pole at $x$. 
Thus, we have established   that the
Green function lies in $BMO_D( \Omega)$, (\ref{ItzaBMO}).

If we choose any sequence $ \{\rho_k  \}$ with 
$ \lim _{ k\rightarrow \infty } \rho _k  =0$, the above argument, applied to 
the rows of $G_{\rho _k } $,  gives a 
subsequence of $ \{ G _{\rho _k} ( x, \cdot )\}$ which converges to a
Green function. As the Green function is unique, the limit must be the function 
$G(x, \cdot)$.  This implies that the entire family  
$\{ G_ \rho (x, \cdot )\}_\rho$ converges
to $G(x, \cdot) $ in the weak-* topology of $BMO_D( \Omega)$.

Next, we recall that $  BMO_D( \Omega) \subset L^ p ( \Omega) $ for any $ p < \infty$. 
Thus we may use the 
 Caccioppoli 
inequality (\ref{Caccioppoli}),  to  conclude 
that 
$$
\int _ { \Omega \setminus \locdom x r } |\nabla _y G_\rho (x, y  ) |^2 \, dy 
\leq C(r) , \qquad \rho < r / 2.
$$
This estimate will also hold for the limit and thus we obtain the
conclusion (\ref{GradientBound}) and that the rows of  $G(x, \cdot )$ are solutions
of $L^* G ^{ \alpha \cdot } ( x, \cdot ) = 0 $  in $ \Omega \setminus \{ x \}$. More 
precisely, we have 
$ A(\phi , G^ { \alpha \cdot } (x, \cdot ) ) = 0 $
whenever $\phi \in \sobolev 2 1 _D ( \Omega)$ and $ \phi$ vanishes in a 
neighborhood of $x$.  Since $G(x, \cdot)$ is a solution of the adjoint equation, we
have the estimates of Corollary \ref{HolderCor}  in $ \Omega \setminus \{ x \}$.

We show how to use these estimates to obtain the pointwise bounds of the Theorem. 
If $ x $ and $y$ are in $ \Omega$, 
we may find a chain of domains 
$ \Omega _j = \locdom  { y_j } { \rho_j}$ for $j=0, \dots, N$
 so that
a) $\Omega _j \subset \Omega \setminus \{ x \} $, 
$ \Omega _j \cap \Omega _ { j+1 } \neq \emptyset $ 
for $j = 0, \dots, N-1$, b) $ \rho_0 \geq c  |x-y |$, $y _0 = y $,  
$\rho _N = r _0/2$, c) $ 1 \leq \rho _{ j + 1 } / \rho _j \leq 2$,  and 
d) $ N \leq C \log ( d/|x-y|)$.  
Since $ G(x, \cdot )$ is in $BMO_D( \Omega)$,   we   have that 
$$
\left | \average _ { \Omega _ j } G(x, y ) \, dy - 
\average _ { \Omega _ { j+1} } G(x, y ) \, dy \right | 
\leq C\| G( x, \cdot ) \| _ { *, D}.
$$
Since the rows of $G(x, \cdot)$ are solutions of the equation $ L^*u =0$, 
the properties of the chain $\{ \Omega _j \}$ and the 
bound (\ref{uBound}) implies the pointwise bound (\ref{LogBound}) for $G$. To 
obtain the H\"older continuity (\ref{GreenHolder}), 
we use that $G$ lies in  $BMO_D( \Omega)$ and the 
local H\"older estimate (\ref{Holder}). If we fix $x$ and $y$, then we may 
apply the local H\"older estimate on a local domain $\locdom y \rho$  with radius 
$\rho$ 
comparable to $ |x-y|$. The boundary estimate (\ref{Boundary}) 
follows immediately from
(\ref{GreenHolder}). 

Next, we turn to the symmetry property of the Green function (\ref{Symmetry}). 
We let $ G_ \rho$ and $ \tilde G_\rho$ be the approximate 
Green functions for $L$ and
$L^ *$ used in the construction of the Green function. 
Using the Green identity (\ref{Dual}) we obtain 
$$
 \average _ { \locdom y \rho } G_ \rho  ^ { \alpha \beta } (x, z) \, dz 
= \average _ { \locdom x \rho } \tilde G_ \rho  ^ { \alpha \beta} (y,z ) \, dz.
$$
Using the H\"older continuity of $G_\rho$ and $ \tilde G_ \rho$ in the second
variable and the Arzela-Ascoli theorem, we may extract a subsequence which
converges uniformly on compact subsets of $ \Omega \setminus \{x\}$. 
Letting $ \rho $ tend to zero, we obtain $G^ { \alpha \beta} (x,y) 
= \tilde G^ { \beta \alpha} 
(y,x) $. 

Finally, to obtain the representation formula of the Theorem for solutions
 with $f$ and 
$ f_N$ not zero,  we may use the Arzela-Ascoli theorem to find a sequence
$ G_{ \rho_j } (x,\cdot )$ which converges uniformly 
on $ \partial \Omega$. Thus, we
may take the limit in (\ref{RepApprox}) to obtain the representation formula
of the Theorem. The convergence of the integral on $ \Omega$ follows since 
$L^ p ( \Omega) \subset H^1 _D ( \Omega)$ and $G_ \rho(x, \cdot)$ converges weakly in  $ BMO_D ( \Omega)$. 
\end{proof}

\note
{ 
Kim  noted that the proof of the representation formula was not
so clear. I hope it is better now. 

When $D \neq \emptyset$, then $L^p \subset H^1_D$, $p>1$. Thus, we have that
$ \int G_ {\rho_k}(x,y) f(y ) \,d y $ converges to $\int G (x,y) f(y) \, dy$ 
and this follows from the weak-* convergence.

A couple of loose ends:

a) Do we have the Lorentz space estimate, $\nabla G(x, \cdot)\in L^ {2, \infty}$?

To do: Study Dolzmann and Muller. They claim to prove Lorentz space estimates.

}

\section{The Green function for the Neumann problem}
\label{Neumann}

In this section we consider the Green function for the Neumann problem (which is 
the mixed problem in the extreme case where $ D = \emptyset$). 
Most of our arguments 
parallel the construction of the Green function for the mixed problem. 
However, there 
is an additional complication. The homogeneous Neumann problem for $L $ has
non-trivial solutions. Hence, we need to impose compatibility conditions on the
data and  conditions to guarantee uniqueness of solutions. It seems that the
most natural condition for uniqueness involves the boundary values of solutions. 

We let $L$ be an operator as defined in  (\ref{OpDef}) and we consider the form  
$ A$ now defined on $ \sobolev 2 1 _ \emptyset ( \Omega)$,   the homogeneous 
Sobolev space of 
functions with one derivative in $L^2 ( \Omega)$. Since we have chosen to 
norm this 
space by the expression $( \int _ \Omega |\nabla u |^2 \, dy ) ^ { 1/2}$, 
the elements of this space will be equivalence classes of functions under the 
equivalence relation $ u $ is equivalent to $v$ if $ u-v $ is a constant. 

We let 
$ {\cal V} = \{ v : A(v, \phi )=0  \mbox{ for all } \phi \in 
\sobolev 2 1 _ \emptyset ( \Omega)\}$ and $ {\cal V ^ *} = \{
v : A(\phi , v ) = 0 \mbox { for all } \phi \in \sobolev 2 1 _ \emptyset 
( \Omega)\}$ denote the solutions of the homogeneous Neumann problems for $L $ and 
$L^*$ respectively. Under our ellipticity assumption (\ref{Ellipticity}) and 
boundedness of the coefficients (\ref{CoeffBound}) 
we have that these spaces are finite dimensional 
and the Fredholm alternative implies that $ \dim {\cal V} = \dim {\cal V}^*$. The 
H\"older estimate of Corollary \ref{HolderCor} implies that the elements of 
these spaces are H\"older continuous. 

We give the weak formulation of the Neumann problem for $L$
\begin{equation} 
\label{WNP}
\left \{ \begin{array}{ll} 
A(u, \phi) = - \langle f , \phi \rangle + \langle f_N,  \phi 
\rangle_ { \partial \Omega},
 \qquad
& \phi \in \sobolev 2 1 _ \emptyset ( \Omega)\\
u \in \sobolev 2 1 _ \emptyset ( \Omega ) \\
\int _ { \partial \Omega } u^ \alpha v ^ \alpha \, d\sigma =0 , \qquad 
& v \in { \cal
V}. 
\end{array}
\right. 
\end{equation}
If $v$ lies in $ \cal V ^ *$, then we have $A(u,v) =0$ for $ u \in \sobolev 2 1 _ 
\emptyset ( \Omega)$. Thus, if we are to find a solution to  (\ref{WNP}),
we must have that $f$ and $f_N$ satisfy the compatibility condition
\begin{equation}
\label{CC}
 \langle f_N , v \rangle _ { \partial \Omega}- \langle f, v \rangle = 0, 
\qquad v \in \cal V ^ *.
\end{equation}
For the operators we consider,  the constant functions lie in $ \cal V ^*$. 
Thus,  the first line of (\ref{WNP}) is satisfied if $ \phi$ is any 
representative of an equivalence class in $ \sobolev 2 1 ( \Omega)$.

We will use the following technical result when we construct the Green function.
\begin{proposition}
\label{AnotherNorm}
Suppose $ \Omega$ is a Lipschitz domain and  that the operator $L$ satisfies
 (\ref{CoeffBound}), (\ref{Ellipticity}), and the coerciveness condition 
 (\ref{Coerce}) with $ D = \partial \Omega$.

a) The norms $\| \cdot \| _ { L^ 2 ( \Omega)}
 + \| \cdot \| _ { \sobolev 21 _ { \emptyset} ( \Omega)} $ 
and $ \| \cdot \| _ { L^ 2 ( \partial \Omega ) } $ are equivalent 
on the space $ \cal V$. 

b) If $ \mu = ( \mu^1, \dots, \mu ^ m) $ is a finite $\reals^m$-valued 
 Borel measure on $\bar \Omega$, then we may find 
$\lambda _ \mu  \in \cal V $  so that
$$
\int _ { \bar \Omega } v^\alpha\, d\mu^\alpha  
  = \int _ { \partial \Omega } \lambda _ \mu ^ \alpha v^ \alpha \, 
d\sigma , \qquad v \in \cal V.
$$
We have the estimate 
 $ \| \lambda _ \mu \| _ { L^ 2 ( \partial \Omega)}
\leq C \| \mu \|$  where $ \|\mu \| $ denotes the total variation of $\mu$. 
\end{proposition}

\begin{proof} Since $ {\cal V} \subset \sobolev 2 1 _ \emptyset 
(  \Omega)$, 
it follows that 
$ \int _ { \partial \Omega } |v|^ 2 \, d\sigma $ is finite on $ \cal V$. 
If $ v \in \cal V$ and $ \int _ { \partial \Omega } |v|^ 2\, d\sigma = 0$, then 
$v =0$ on 
$ \partial \Omega$ and it follows that $ v$ is in the Sobolev space 
$\sobolev 2 1 _ { \partial \Omega } ( \Omega)$. Since we assume that the 
form $A$ is coercive on this space, it follows that 
$ v= 0$. Thus, we have $ \| \cdot \| _ { L^ 2 ( \partial \Omega)}$ is a
norm on this space. Since $ \cal V$ is finite dimensional, it follows that 
the norms $ \| \cdot \|_ { L^2 ( \Omega)} + \| \cdot \| _ { \sobolev 2 1 _ \emptyset ( \Omega) }$ and $\| \cdot \| _ {L^ 2( \partial \Omega)}$ are
equivalent. 

To establish  part b), observe that the local boundedness estimate
  (\ref{uBound}) 
implies that  $ v \rightarrow \int _ { \bar \Omega } v^ \alpha \, d\mu^\alpha$ 
is a 
continuous linear functional on $ \cal V$. According to part a), $ \cal V$ 
is a Hilbert 
space under the inner product  $\int_ { \partial \Omega} 
 u^\alpha  v^ \alpha \, d\sigma$. Hence, the
Riesz representation implies that we have a unique $ \lambda _ \mu \in \cal V$ so that
$$
\int _ { \bar \Omega } v^ \alpha\, d\mu^ \alpha
 = \int _ { \partial \Omega } 
 \lambda _ \mu ^ \alpha v^ \alpha \, 
d\sigma .
$$
Corollary \ref{HolderCor} implies that the elements of $ \cal V$ are bounded 
functions, the estimate for $ \lambda _ \mu$ follows
from the Riesz representation theorem. 
\end{proof} 

Next, we give a standard existence theorem for the weak Neumann problem (\ref{WNP}). 

\begin{theorem}
\label{SolveN}
 Let $ \Omega$ be a Lipschitz domain and suppose that $L$ satisfies 
the ellipticity condition (\ref{Ellipticity}), has bounded 
coefficients (\ref{CoeffBound}),
and  that (\ref{Coerce}) holds for $D= \partial \Omega$. 
If $ f$ and $f_N$ satisfy the compatibility condition (\ref{CC}), then the 
weak Neumann problem
(\ref{WNP}) has a unique solution and the solution $u$ satisfies the estimate 
$$
\| u \| _ { L^ 2 ( \Omega ) } 
 + \| u \| _ { \sobolev 2 1 _ { \emptyset } ( \Omega)}
\leq C ( \| f \| _ { \sobolev 2  { -1} _ \emptyset ( \Omega ) } 
+ \| f_N \| _ { \sobolev 2 { -1/2 } _ \emptyset ( \partial \Omega ) } ) .
$$
\end{theorem}

\begin{proof} The existence of solutions is a standard consequence of
 the Fredholm alternative. See Gilbarg and Trudinger 
 \cite[Theorem 8.6]{GT:1983} for 
 the proof of a similar result. To establish uniqueness, if $u$ solves 
 (\ref{WNP}) with
 $f=0$ and $ f_N =0$, then we have that $ u \in \cal V$. Part a) of Lemma
  \ref{AnotherNorm}
and the condition 
$ \int _ { \partial \Omega } u^ \alpha  v ^ \alpha\, d\sigma = 0$
for $ v \in \cal V$ imply that $u =0 $.
\end{proof}

We now give estimates for solutions of the Neumann problem to take the place of
Corollary \ref{TwoEstimates}. Since the homogeneous Neumann problem may have 
non-constant solutions, Corollary \ref{TwoEstimates} cannot hold for all solutions of 
the Neumann problem. 

We consider a function $f$ in $L^r( \Omega)$, $r>1$, 
and use Lemma \ref{AnotherNorm} to find a 
function $ \lambda _f \in {\cal V} ^ *$ so that the pair $f$ and $ \lambda _f |_ 
{\partial \Omega}$ satisfy the compatibility condition needed to solve the Neumann
problem,
\begin{equation}
\label{ACC}
\langle  \lambda _f , v \rangle_{ \partial \Omega} - \langle f ,v \rangle=0 , \qquad 
v\in \cal V ^*.
\end{equation} 
We let $ \Lambda _f $ be defined by 
$ \Lambda _f ( \phi) = 
 \int _ {\partial \Omega } 
\lambda _f ^ \alpha \phi ^ \alpha\, d\sigma - 
\int _ \Omega f^ \alpha \phi ^ \alpha\, dy.
$
We observe that 
the constant functions lie  in $ \cal V^*$ and hence we have  $\Lambda _f ( c ) =0$ 
for all 
constant vectors $c\in \reals ^m$.  The elements in $ \cal V^*$ are bounded and 
for all $p$,   $1\leq p< \infty$, 
we have the  inequalities
\begin{equation}
\label{Homogeneous}
r_0^ { -2/p}\| u - \bar u \| _{ L^ p( \Omega) } + r_0^ { -1/p}
\| u - \bar u \|_{ L ^ p ( \partial 
\Omega ) } \leq C\| \nabla u \|_ {L^ 2 ( \Omega)},
\end{equation}
where $ \bar u = \average _ \Omega u \, dy$. 
It follows that $ \Lambda _f$ lies in $ \sobolev { 2 } { -1}( \Omega)$ and thus the 
solution of the Neumann problem
\begin{equation}
\label{NLp}
\left\{
\begin{array}{ll}
A(u, \phi ) = \Lambda _f (\phi ) , \qquad &\phi  \in \sobolev 2 1 ( \Omega) \\
u \in \sobolev 2 1 _ \emptyset ( \Omega) \\
\int _{ \partial \Omega } u ^ \alpha v^ \alpha \, d\sigma =0 , \qquad & v \in {\cal V}. 
\end{array}
\right. 
\end{equation}
exists and satisfies $\|\nabla u \| _ { L^ 2( \Omega) } 
\leq  C(r, r_0)\| f\| _ { L^ r( \Omega)}$.  
Next, we consider the adjoint problem for $ g\in L^ r ( \Omega)$, 
\begin{equation}
\label{ANLp}
\left \{
\begin{array}{ll}
A(\phi, w) = \Lambda^* _g (\phi ) , \qquad &\phi  \in \sobolev 2 1 ( \Omega) \\
w \in \sobolev 21  _ \emptyset ( \Omega) \\
\int _{ \partial \Omega } w ^ \alpha v^ \alpha \, d\sigma =0 , 
\qquad & v \in {\cal V^*}. 
\end{array}
\right.
\end{equation}
where  $\Lambda _g  ^* (\phi) = \langle \lambda _g, \phi \rangle_ { \partial \Omega}
- \langle g , \phi \rangle $ and  $ \lambda _g\in \cal V$ satisfies  $ \int _ \Omega v^ \alpha g^ \alpha \, dy 
=\int _{ \partial \Omega } v ^ \alpha \lambda _g ^ \alpha \, d\sigma $, $ v\in
\cal V$.  Similar considerations give the existence of a solution to the
adjoint problem with $\| \nabla w \|_{L^2(\Omega)} \leq 
C \| g \| _ { L^ r ( \Omega) }$. 
Furthermore, from (\ref{Dual}) we obtain
\begin{equation}
\label{NDual}
 \int_ \Omega u ^ \alpha g^ \alpha\, dy = 
\int_ \Omega w ^ \alpha f^ \alpha\, dy.
\end{equation}
The boundary integral 
$ \int_ { \partial \Omega}  \lambda _f ^ \alpha w^ \alpha\, d\sigma$ vanishes 
since the Neumann data  $ \lambda _f $ lies in $  \cal V^ *
$ and $w$ is perpendicular to this space in $L^2( \partial \Omega)$. The other
boundary integral vanishes for similar reasons.

\begin{corollary}
\label{NTwoEstimates} 
Let $ t \in ( 2, t_0)$ with $t_0$ as in Theorem \ref{CaccTheorem} 
and $r$ be defined by 
$ 1/r = 1/2 + 1 / t$. 
If $ f \in L^ r ( \Omega)$,  then the solution of 
the weak Neumann problem (\ref{NLp}) satisfies the estimates
\begin{equation}
\label{NPrimalEstimate}
\|\nabla u  \|_ { L^t ( \Omega )} \leq C \| f \| _ { L^ r ( \Omega)}
\end{equation}
and 
\begin{equation}
\label{NDualEstimate}
\|  u \| _ {L^ { r' }( \Omega) } 
\leq C \| f \|_{\sobolev { t'} {-1} ( \Omega)}.
\end{equation}
\end{corollary}

\begin{proof} With the work above the proof is the same as Corollary \ref{TwoEstimates}.
The proof of Corollary \ref{TwoEstimates} fails at the first line, because we do not
have a Poincar\'e inequality. Since $ \Lambda _f ( u) = \Lambda _f (u-\bar u)$, we
may use the estimates (\ref{Homogeneous}) to show that $ \Lambda_f$ lies in
the dual of $ \sobolev 2 1 _ \emptyset( \Omega)$.  With $ \Lambda _f \in \sobolev 2 1 _ \emptyset ( \Omega)$,  the proof of this Corollary
is identical to the proof of Corollary \ref{TwoEstimates}. 
\end{proof}

\begin{remark} For the second estimate of Corollary \ref{NTwoEstimates} 
to be useful, we must have 
that $f$ lies in $ \sobolev {t ' } {-1} ( \Omega)$.
In particular, we must have that the mean-value of $f$ is zero. 
\end{remark} 

\begin{corollary}
\label{NAtomicCorollary}
 If $a$ is an atom for the Hardy space 
  $H^1 _ \emptyset( \Omega)$, then the solution
of (\ref{NLp}) with $f=a$ satisfies
$$
\| u \| _ { L ^ \infty ( \Omega) } \leq C.
$$
\end{corollary}

\begin{proof} Given the estimates Corollary \ref{NTwoEstimates}, the result follows 
as in Corollary \ref{AtomicCorollary}. 
\end{proof}

We are ready to define a Green function for the the Neumann problem. We say that 
$ G(x, \cdot)$ is a {\em Green function for the Neumann problem with pole at }$x$ if 
$G(x, \cdot) $ is in $L^ 1 ( \Omega)$ and whenever $a$ is an atom and $u$ the
corresponding
solution to (\ref{NLp}) with $f=a$, then 
we have
$$
u ^ \alpha ( x) = \int _ { \Omega } G^ { \alpha \beta } ( x, y ) a ^ \beta ( y) \,dy.
$$
It is clear that the Green function is unique up to a constant.

\begin{theorem} 
\label{NGreenTheorem}
Let $ \Omega $ be a Lipschitz domain and 
suppose that $L$ satisfies 
the ellipticity condition (\ref{Ellipticity}), has bounded 
coefficients (\ref{CoeffBound}),
and the form for $L$ 
is coercive on $ \sobolev 2 1 _ { \partial \Omega } ( \Omega )$. 
Then there exists a unique Green function $G(x,y)$ which satisfies the following 
estimates
\begin{eqnarray}
\label{BMON} 
 && G(x, \cdot ) \in BMO _ { \emptyset} ( \Omega)
\\
&& \nabla _y G(x, \cdot ) \in L^2 ( \Omega \setminus \locdom xr ) , 
\qquad \mbox{for } x\in \Omega, \ r>0
\\
\label{LogBoundN}
 &&
|G(x,y ) | \leq C ( 1+ \log ( d/|x-y |) ) , \quad x , y \in \Omega
\\
\label{HolderBoundN} 
 &&
|G(x,y ) - G(x, z ) | \leq C  \frac { |y - z |^\gamma} { |x-y |^\gamma} 
, \quad x,y,z \in \bar \Omega, \  |y-z| < \frac 1 2 | x- y |.
\end{eqnarray}
If $G$ and $\tilde G$ are the Green functions
 for $L$ and $ L^*$, respectively, we 
may find  representatives of $G$ and $ \tilde G$ so that
\begin{equation}
\label{SymmetryN}
G^ { \alpha \beta} (x,y)  = \tilde G^ { \beta\alpha}
(y,x) .
\end{equation}
Furthermore,  if  $u$ is a solution of the weak Neumann problem with 
$ f \in L^ p ( \Omega)$ 
and $f_N \in L^ p ( \partial \Omega )$ for some $p > 1$ and satisfying the
compatibility condtion (\ref{CC}), then we have 
\begin{equation}
\label{RepN}
u ^ \alpha (x) = \int _ \Omega G^ { \alpha \beta } ( x,y ) f^ \beta (y) \, dy
- \int _ { \partial \Omega }  G^ { \alpha \beta } ( x,y ) f^ \beta _N(y)
\, d\sigma, \qquad \alpha = 1, \dots, m. 
\end{equation}
\end{theorem}

\begin{proof} 
We fix $ x\in \Omega$,  $ \alpha \in \{ 1, \dots, m\}$, $ \rho$ with 
$0<\rho < r_0$, 
and let 
$ g =  {e _ \alpha }  \chi _ { \locdom x \rho }/ { | \locdom x \rho |} $
where 
$e_\alpha$ is the unit vector in the direction of the $\alpha$th coordinate axis. 
We let $ \lambda^{ \alpha \cdot} _ \rho \in \cal V$ be chosen so that 
$g$ and $\lambda^ { \alpha \cdot} _\rho|_{ \partial \Omega}$ 
satisfy the compatibility condition needed
to solve the adjoint problem  (\ref{ANLp}). 
We let $G^ { \alpha \cdot }_ \rho (x, \cdot ) $ be the solution of 
   (\ref{ANLp}) with 
$g$ as above and $g_N$ replaced by $ \lambda^ {\alpha \cdot}
 _ \rho$. We let $u$ be a solution of the Neumann problem
 (\ref{WNP}) with data $f$ and $f_N$ in $ L^ p ( \Omega)$ 
 and $L^{ p}(\partial \Omega)$,  $p>1$, 
 respectively. From  (\ref{Dual}), we 
have 
\begin{equation}
\label{RepApproxN}
\average _ { \locdom x \rho } u ^ \alpha( y ) \, dy 
= \int _{ \partial \Omega  } G^ { \alpha \beta }_\rho( x, y ) f_N ^ \beta (y) \, dy -
\int _ \Omega G_\rho^ { \alpha \beta } ( x, y ) a ^ \beta (y) \, dy
\end{equation}
If we let $f$ be an atom and $f_N = 0$, then  Corollary \ref{NAtomicCorollary} 
implies
the left-hand side  of (\ref{RepApproxN}) is bounded by 
a constant that is  
independent of $a$. It follows that 
$\| G_ \rho ^ { \alpha \cdot}(x, \cdot)\|_ {* , \emptyset}  \leq C$  with 
$C$ independent of $ \rho$. 
 We may use compactness in the finite dimensional
space $\cal V$ and the Banach-Alaoglu theorem to find $ \lambda ^ { \alpha \cdot }
\in \cal V$, 
$G^ { \alpha \cdot } ( x, \cdot ) \in BMO _ \emptyset ( \Omega)$, and a sequence
$ \{ \rho _j \} _ { j =1 } ^ \infty$ with $ \lim _ { j \rightarrow \infty }
\rho _j = 0$ so that 
$ G_ { \rho _j } ^ { \alpha \cdot }( x, \cdot) $ 
converges weakly to $ G^ { \alpha \cdot } ( x, \cdot )$ in $ BMO_ \emptyset( \Omega)$
and $\lambda ^ { \alpha \cdot } _{ \rho_j}$ converges to $ \lambda ^ { \alpha \cdot}$ in 
$L^ 2( \partial \Omega)$.
Combining this weak convergence of $G_{\rho_j}$  in $BMO_ \emptyset ( \Omega)$ 
with the H\"older continuity of $u$, we
obtain that 
$$
u^ \alpha(x) = \int _ {\Omega} G^ { \alpha \beta } (x, y ) a^ \beta ( y) \, dy.
$$
Thus, we have found our Green function and we have 
(\ref{BMON}). From the uniqueness  for the solution $u$,
it follows that the limit $ G(x, \cdot)$ is unique and thus we have that 
$ G_\rho$ converges
for all $ \rho$ and not just a  subsequence. 

Using the Caccioppoli inequality (\ref{Caccioppoli}) we can show that 
$$
\int _ { \Omega \setminus \locdom x r } |\nabla G_ \rho (x,y ) |^ 2\, dy 
\leq C(r) , \qquad \rho < r/2.
$$
As this estimate is uniform  in $ \rho$ for $ \rho $ sufficiently small, 
we obtain
that  for all $ \phi \in \sobolev 2 1 _ \emptyset ( \Omega)$ 
which vanish in neighborhood
of $x$ that 
$$
\int _ \Omega a^{ij}_{\gamma \beta}  \frac { \partial \phi ^ \beta}{\partial y _ j }
\frac { \partial G^ { \alpha \gamma }}{ \partial y_i } (x, y )   
\, dy = \int _ {\partial \Omega }\lambda ^ { \alpha \gamma } \phi^ \gamma \, d\sigma.
$$

The pointwise estimates (\ref{LogBoundN}) and (\ref{HolderBoundN}) 
follow from (\ref{BMON}) and the local estimates in Corollary \ref{HolderCor}.
As $G$ has non-zero Neumann data, it is important that these results allow non-zero
Neumann data. 
The 
argument is identical to that of Theorem \ref{GreenTheorem}.  

We establish the symmetry property (\ref{SymmetryN}). 
 We let $ G_\rho$ and $ \tilde G_ \rho$ be 
approximate Green functions for $L$ and $L^*$ as defined above. 
 From  (\ref{NDual}), we obtain  that
\begin{equation*}
\average _ {\locdom x \rho } \tilde G_ \rho^ { \beta \alpha }( y , z) 
 \, dz 
= 
\average _ {\locdom y \rho }  G^ {\alpha \beta }_\rho( x, z)  \, dz. 
\end{equation*}
Now, we may let $ \rho$ tend to zero and obtain (\ref{SymmetryN}).

Next, we claim that the mean value $ \int _ \Omega 
G _\rho^ { \alpha \beta } ( x, y ) \, dy $
is bounded for all $ \rho$. 
As a first step, let $ u$ be the solution of (\ref{NLp}) with $ f = e _
\beta \chi_ \Omega$. According to Corollary \ref{NTwoEstimates}, the
solution  $u$  lies  in $L^r ( \Omega)$ for some  $r$.  As the Neumann
data $ \lambda _f$ lies in ${ \cal V ^ *}$ and hence is bounded, we
may use the estimates of Theorem \ref{CaccTheorem} to conclude that
$\nabla u$ lies in $L^t( \Omega)$ for some $ t>2$.  Since $u $ lies in
$L^ r( \Omega)$ and $ \nabla u $ lies in $L^t ( \Omega)$, we may
conclude that $u$ is bounded. We apply 
 (\ref{Dual}) and obtain 
$$
\average _ { \locdom x \rho } u ^ \alpha\, dy = \int _ \Omega G_\rho^ { \alpha \beta } 
(x,y) \, dy .
$$
Since $u$ is bounded,    the claim follows. 
Since $ G_ \rho ^ { \alpha \cdot}( x, \cdot)$ is bounded in 
$BMO_ \emptyset( \Omega)$ and
the mean values are bounded, it follows that a subsequence of 
 $G^ { \alpha \cdot}_ \rho(x\cdot)$ 
converges weakly in $L^p( \Omega)$ for each $p $ finite. In addition, 
$G_\rho^ { \alpha \cdot}(x, \cdot)$ is a solution of $L^ * w=0$ in a 
neighborhood of the boundary
and the Neumann data $ \lambda _ \rho$ lies in ${\cal V}$ and hence is bounded. 
Thus, we
may extract a subsequence $\{ G_{ \rho _j } (x, \cdot)\}$ 
 which converges uniformly on $ \partial \Omega$. 
We may let 
$ \rho \rightarrow 0 ^ + $ in the representation formula (\ref{RepApproxN}) and
use  the
continuity of $u$ to obtain the representation formula (\ref{RepN}). 

\end{proof}

\section{The Green function in the plane} 
\label{Free}

In this section, we define a Green function in the plane. We will work in the
homogeneous  Sobolev  
space $ \sobolev 2 1 ( \reals ^2)$ which consists of functions $ \phi $ 
with $ \nabla \phi \in L^ 2 ( \reals ^ 2)$. We  norm this space with
$$
\| \phi \| _ { \sobolev 2 1 ( \reals ^ 2 ) } = \left ( \int _ { \reals ^ 2} 
|\nabla \phi |^ 2 \, dy \right) ^ { 1/2} .
$$
and the elements of this Hilbert space will be equivalence classes of
functions under the relation $u$ is equivalent to $v$ 
if $ u - v $ is constant.  In $ \reals ^2$,
the local domains $ \locdom x \rho $ are disks and  we will use the more
standard $ \ball x \rho = \{ y : |x-y | < \rho\}$ to denote these disks. 

We assume that the form $A$ is coercive in the sense that 
\begin{equation}
\label{CoerceF}
A(u,u ) \geq M^ { -1} \int _ {\reals^2}  |\nabla u |^ 2 \, dy , \qquad u \in 
\sobolev 2 1 ( \reals ^2).
\end{equation}
We may see that this condition holds for the Lam\'e system using an approximation 
argument and integration by parts.

It is an immediate consequence of the Lax-Milgram theorem that the weak formulation 
of the problem in the plane
\begin{equation} 
\label{FP}
\left \{ 
\begin{array} {ll}
A(u, \phi ) = -\langle f, \phi\rangle , \qquad & \phi \in \sobolev 2 1 ( \reals ^2 ) 
\\
u \in \sobolev 2 1 ( \reals ^ 2)
\end{array}
\right. 
\end{equation}
has a unique solution when $f$ is in the dual of $ \sobolev  2 1 ( \reals ^ 2)$, 
$\sobolev 2 { -1} ( \reals ^2)$.

We will approximate $u$ by considering the Dirichlet problem in disks
$ B_R = \{ x  :|x|< R \}$ for $R>0$. 
We let $a$ be an atom for $ \reals ^2$ that is supported in $ \ball x \rho$
and for $R$ large, we let $u_R$ be the solution of the Dirichlet problem
$$
\left \{ \begin{array}{ll}
A(u_R, \phi ) = -\int_ \Omega a^ \alpha \phi ^ \alpha \, dy \qquad & \phi  \in
 \sobolev 
2 1 _ { \partial B _R} ( B_R ) \\
u_ R \in \sobolev 2 1 _ { \partial B_R} ( B_R) 
\end{array}
\right.
$$
From Proposition \ref{EnergyProp}  and Corollary \ref{AtomicCorollary},  
we have that 
$
\| u _R \|_{ L^ \infty(B_R)} +
\| \nabla u_R \| _ { L^ 2 ( B _ R )}  \leq C$ and from Corollary 
\ref{TwoEstimates} and Morrey's inequality (\ref{Morrey})
\begin{equation}
\label{ApproxBounds}
  \| u _ R \| _ { L^ { r' } ( B_R)}
+ \sup _{ x\neq y } \frac {|u_R(x) - u_R(y) |} { |x-y|^ \gamma} 
\leq C ( \rho ). 
\end{equation}
The H\"older index $ \gamma = 1 -2/t $  with  $t$ and $ r'$ 
as in Corollary \ref{TwoEstimates}.  The estimates of Corollary \ref{TwoEstimates} 
are scale invariant and thus hold uniformly in $R$. The dependence on $ \rho $ arises
because the norm of an atom in $L^r(B_R)$ 
and $ \sobolev {t '}{-1}(B_R)$ will depend on $\rho$. 
Thus, we  have a function $u \in \sobolev 2 1 ( \reals ^ 2)$  so that 
$\lim _ { R\rightarrow \infty} u_R = u$ weakly in 
 $ \sobolev 2 1 ( B_S) $ for each $ S > 0$
and $u$ solves (\ref{FP}) with $ f= a$. Note that since the limiting function
$u$  is unique, we have convergence for the entire family, not just a subsequence. 
Furthermore from the Rellich compactness theorem, we have that $u_R $ converges 
in $L^ { r'} ( B_S) $ for each $S > 0$.  The H\"older estimate in (\ref{ApproxBounds})
 and the
Arzela-Ascoli theorem  imply that we also have that 
 $u_R$ converges locally uniformly to $u$. 
According to Corollary \ref{AtomicCorollary}, 
the functions $u_R$ are uniformly bounded, hence the same holds for $u$. 
Thus, if $ f=a$, an atom, the solution of (\ref{FP}) may be chosen so that
\begin{eqnarray}
\label{FPBound}
\| u \|_ { \sobolev 2 1 ( \reals ^ 2)} + \|u \|_{L^ \infty( \reals^2)} 
& \leq &  C \\
\label{FPExtras}
\| u \|_ { L^ { r' } ( \reals ^ 2 )} + 
\sup _{ x\neq y } \frac {|u(x) - u(y) |} { |x-y|^ \gamma} & \leq & 
C (\rho).
\end{eqnarray}

We give a definition of the Green function in the plane. We say that 
$G(x, \cdot )$
{\em is 
a Green function in the plane for (\ref{FP}) with pole at $x$} if $ G(x,\cdot) $ 
is in $L^ 1 _ { loc } ( \reals ^ 2 ) $ and for each atom $a$, the solution of 
(\ref{FP}) is given by 
\begin{equation}
\label{FPRep}
u^ \alpha(x) 
= \int _ { \reals ^ 2 } G^ {\alpha\beta}(x, y) a^ \beta(y) \, dy, \qquad \alpha 
=1, \dots,m .
\end{equation}
Since solutions of this weak problem are unique, it is immediate that for each
$x$,   $G(x, \cdot ) $ 
is a unique element of $BMO ( \reals ^2 )$. In other words, $G(x, \cdot )$ 
is unique
up to a constant. 

Finally, we give a theorem which establishes existence and regularity
of the Green function in $ \reals ^2 $. 

\begin{theorem} 
\label{FGreenTheorem}
If $L$ satisfies  (\ref{CoeffBound}) and (\ref{CoerceF}), then there is a 
unique Green function $G(x, \cdot)$ with pole at $x$ and the 
Green function satisfies
\begin{gather}
G(x, \cdot ) \in BMO ( \reals^2) \\
 \nabla_ y G(x, \cdot ) \in L^ 2 _ { loc} ( \reals ^ 2
\setminus \{ x \} )\\
\label{LogF}
 |G(x,y )- \average _{ \ball x 1 } G(x,z)\, dz | \leq C ( 1 + | \log ( |x-y |)| ), 
\qquad x, y \in \reals^2 \\
\label{HolderF}
  | G(x,y ) - G( x, z ) | \leq C \frac { |y -z | ^ \gamma } 
{ |x-y | ^ \gamma } , \qquad \mbox{if } | y - z | < \frac 1 2 | x - y | 
\end{gather}
The exponent $ \gamma$ is as in Corollary \ref{HolderCor} and 
$C$ may be chosen to depend on $M$ and   $\gamma$. 

Finally, if $G$ and $\tilde G$ are the Green functions in the plane 
for the operators
$L$ and $L^*$, then we may find representatives which satisfy the 
symmetry condition
\begin{equation}
\label{FSymm}
G^ { \alpha \beta } ( x,y ) = \tilde G ^ { \beta \alpha } ( y,x) .
\end{equation}
\end{theorem}

\begin{proof} 
To construct the Green function $G(x, \cdot)$ we fix $x$, $ \rho>0$,   and let 
$f_ \rho = \frac 1 { \pi \rho ^ 2} \chi _ { B_ \rho (x)} - 
\frac { \rho^ 2} { 3 \pi}\chi_ { B _ {2/\rho } (x) \setminus B_ { 1/\rho } (x) }$. 
 As $ f_\rho e_ \alpha$ is an atom (though the constant depends on $ \rho$), we may
let $G_ \rho(x, \cdot)$ be the weak solution of the equation 
$L^* G_ \rho ^ { \alpha \cdot } ( x, \cdot ) = f_ \rho e_ \alpha$. 
We let  $a$ be an atom
and $u$  the solution of $Lu =a $ constructed above. 
As in the proof of (\ref{Dual}), the weak formulations of the equations satisfied by $ u$ and 
$ G^{\alpha\cdot}_\rho(x, \cdot) $ give 
\begin{equation}
\label{TwoWeak}
A(u , G^ { \alpha \cdot}_ \rho (x, \cdot ) )
= \int _ { \reals ^ 2} G^ { \alpha \beta }_ \rho (x, y)a^ \beta(y) \, dy
= \int _ { \reals ^ 2} f_ \rho u ^ \alpha \, dy.
\end{equation}
The estimate for
$\|u \|_ {L^\infty (\reals^2)}$ in (\ref{FPBound}) implies that 
$ |\int G_ \rho^ {\alpha \beta} (x,\cdot)a^ \beta \, dy |
= |\int f_ \rho u_ \alpha \, dy| 
\leq C$. Now, we may conclude from (\ref{NewNorm})  that 
$G_ \rho(x, \cdot )$ lies in $BMO( \reals ^2)$. 
 Thus, we may use the
Banach-Alaoglu Theorem to find a function $G(x,\cdot)$ and a  
sequence $\{G_{\rho _k } (x, \cdot) \}$  with $ \lim _ { k \rightarrow \infty }
\rho _k = 0 $ so that $ G_ { \rho_k } (x, \cdot)$ 
converges to $G(x, \cdot)$ in the weak-* topology of $ BMO( \reals ^2)$. 
From the estimates in (\ref{FPExtras}), it follows that 
$$
u^ \alpha(x) = \lim _ { \rho \rightarrow 0 ^ +} \int_ {\reals ^2} 
 f_ \rho u^ \alpha\, dy.
$$
Hence, we obtain the representation formula (\ref{FPRep}). 
 We may use Caccioppoli's
inequality (\ref{Caccioppoli})
to obtain uniform bounds on 
$ \nabla _ y G_R ( x, \cdot) $ in $L^2$ of compact subsets of 
$ \reals ^ 2\setminus \{ x \}$. Thus $G(x, \cdot )$ is a solution of 
$L^ * G(x, \cdot) =0 $ in $ \reals ^ 2 \setminus \{ x \}$. Now the 
pointwise estimates (\ref{LogF}) and (\ref{HolderF}) follow as they do
for the mixed problem. 

Finally, we establish the symmetry property (\ref{FSymm}). 
As our construction of $G$, stands we have 
no information about the behavior of  $G$ in the first variable. 
We begin by claiming that we can fix a representative of $G(x, \cdot) $ so that
$G$ is locally integrable in $ \reals ^2 \times \reals ^2$. Towards this
end, we observe that if  we fix $ x$  in $ \reals ^2$ and let $ h _ {\rho, x} = 
\frac 1 { \pi \rho ^2} ( \chi_ { B_ \rho (x) } - \chi _{ B_ \rho (0)}) $,
then $ e_ \alpha h_{\rho,x}$ 
lies in  $ \sobolev  {{ t'}} { -1}( \reals^2) $,  
the dual of $ \sobolev t 1 ( \reals ^2)$, 
$ t\geq 2$, with $\| h_ \rho e _ \alpha \| _ { \sobolev { t' } { -1} ( \reals ^2)}
\leq C |x |^ { 1 - 2/t}$.    We let $ v_\rho(x, \cdot) $ be the solution of 
$L^ * v_\rho( x, \cdot) =  e _ \alpha h _ { \rho, x}$.   
Using the estimates of Corollary \ref{TwoEstimates}, we may show that the 
 map  $ x \rightarrow v_ \rho (x, \cdot) $ is a continuous
map from $\reals ^2$ into $ L^ { r' } ( \reals^ 2) $ where $ r' $ is as in
(\ref{DualEstimate}). 
If we fix a representative of $G_ \rho(0, \cdot) $, the function 
$ G _ \rho (0, \cdot) + v_ \rho (x, \cdot)$ gives an approximate Green function 
that is locally integrable
in $ \reals^2 \times \reals ^2$. If we let $ \rho \rightarrow 0^+$, we obtain
the same conclusion for $G$.  
 
We let $ G$ and $ \tilde G$ be the Green functions as constructed in the previous
paragraph for $L$ and $L^*$. We fix atoms $a$ and $b$ and let $u$ and $v$ solve
the equations $Lu =a $ and $ L^*v=b$. From the weak formulation  (\ref{FP}), we
 have 
 $A(u,v) = \int u^ \alpha b^ \alpha\, dy = \int a^ \beta v^ \beta\, dy$.  Using the
 representation (\ref{FPRep}) and Fubini's theorem, we obtain 
$$
\int_ {\reals^2\times \reals^2}   b^ \alpha (x) G^ { \alpha \beta } (x,y) a^ \beta (y)
 \, dx \, dy 
= \int _{\reals^2\times\reals^2} 
b^ \alpha (x) \tilde G ^ { \beta\alpha} (y,x) a^ \beta(y) \, dx\, dy .
$$
As this holds for all atoms $a$ and $b$, we have 
functions $ \phi$ and $ \tilde \phi $ 
so that 
$$
G(x, y) + \phi (x) =  \tilde G(y,x) + \phi(y).
$$
\end{proof}
 
\note{
 
Write Lemma for solution of FP with atomic data.

Emphasize that results apply to Lame. 

Can we show that the solution operator is continuous from 
$ \sobolev {t'} {-1} $ into
$\sobolev {t'}1$? This would give a more direct proof of the continuity 
of $G$ in the
first variable. 
}


\begin{thebibliography}{10}

\bibitem{MR1600066}
P.~Auscher, A.~McIntosh, and P.~Tchamitchian.
\newblock Heat kernels of second order complex elliptic operators and
  applications.
\newblock {\em J. Funct. Anal.}, 152(1):22--73, 1998.

\bibitem{MR1471017}
M.~Calanchi, L.~Rodino, and M.N. Tri.
\newblock Solutions of logarithmic type for elliptic and hypoelliptic
  equations.
\newblock In {\em Proceedings of the {C}onference ``{D}ifferential
  {E}quations'' ({I}talian) ({F}errara, 1996)}, volume 41, suppl., pages
  111--127 (1997), 1996.

\bibitem{MR1253178}
D.C. Chang.
\newblock The dual of {H}ardy spaces on a bounded domain in {${\bf R}^n$}.
\newblock {\em Forum Math.}, 6(1):65--81, 1994.

\bibitem{MR1190215}
S.~Chanillo and Y.Y. Li.
\newblock Continuity of solutions of uniformly elliptic equations in {${\bf
  R}^2$}.
\newblock {\em Manuscripta Math.}, 77(4):415--433, 1992.

\bibitem{arXiv:1112.2436v1}
J.~Choi and S.~Kim.
\newblock Neumann functions for second order elliptic systems with measurable
  coefficients.
\newblock arXiv:1112.2436v1.

\bibitem{CW:1976}
R.R. Coifman and G.~Weiss.
\newblock Extensions of {H}ardy spaces and their use in analysis.
\newblock {\em Bull. Amer. Math. Soc.}, 83:569--645, 1976.

\bibitem{MR1354111}
G.~Dolzmann and S.~M{\"u}ller.
\newblock Estimates for {G}reen's matrices of elliptic systems by {$L^p$}
  theory.
\newblock {\em Manuscripta Math.}, 88(2):261--273, 1995.

\bibitem{MR2485428}
H.~Dong and S.~Kim.
\newblock Green's matrices of second order elliptic systems with measurable
  coefficients in two dimensional domains.
\newblock {\em Trans. Amer. Math. Soc.}, 361(6):3303--3323, 2009.

\bibitem{FG:1973}
F.~W. Gehring.
\newblock The {$L\sp{p}$}-integrability of the partial derivatives of a
  quasiconformal mapping.
\newblock {\em Acta Math.}, 130:265--277, 1973.

\bibitem{MG:1983}
M.~Giaquinta.
\newblock {\em Multiple integrals in the calculus of variations and nonlinear
  elliptic systems}, volume 105 of {\em Annals of Mathematics Studies}.
\newblock Princeton University Press, Princeton, NJ, 1983.

\bibitem{MR549962}
M.~Giaquinta and G.~Modica.
\newblock Regularity results for some classes of higher order nonlinear
  elliptic systems.
\newblock {\em J. Reine Angew. Math.}, 311/312:145--169, 1979.

\bibitem{GT:1983}
D.~Gilbarg and N.S. Trudinger.
\newblock {\em Elliptic partial differential equations of second order}.
\newblock Springer-Verlag, Berlin, 1983.

\bibitem{MR990595}
K.~Gr{\"o}ger.
\newblock A {$W^{1,p}$}-estimate for solutions to mixed boundary value problems
  for second order elliptic differential equations.
\newblock {\em Math. Ann.}, 283(4):679--687, 1989.

\bibitem{MR657523}
M.~Gr{\"u}ter and K.O. Widman.
\newblock The {G}reen function for uniformly elliptic equations.
\newblock {\em Manuscripta Math.}, 37(3):303--342, 1982.

\bibitem{MR706075}
J.L. Journ{\'e}.
\newblock {\em Calder\'on-{Z}ygmund operators, pseudodifferential operators and
  the {C}auchy integral of {C}alder\'on}, volume 994 of {\em Lecture Notes in
  Mathematics}.
\newblock Springer-Verlag, Berlin, 1983.

\bibitem{KN:1985}
C.E. Kenig and W.M. Ni.
\newblock On the elliptic equation ${L}u-k+{K}\,{\rm exp}[2u]=0$.
\newblock {\em Ann. Scuola Norm. Sup. Pisa Cl. Sci. (4)}, 12(2):191--224, 1985.

\bibitem{KP:1993}
C.E. Kenig and J.~Pipher.
\newblock The {N}eumann problem for elliptic equations with nonsmooth
  coefficients.
\newblock {\em Invent. Math.}, 113:447--509, 1993.

\bibitem{LSW:1963}
W.~Littman, G.~Stampacchia, and H.~Weinberger.
\newblock Regular points for elliptic equations with discontinuous
  coefficients.
\newblock {\em Ann. della Sc. N. Sup. Pisa}, 17:45--79, 1963.

\bibitem{NM:1963}
N.G. Meyers.
\newblock An {$L\sp{p}$}-estimate for the gradient of solutions of second order
  elliptic divergence equations.
\newblock {\em Ann. Scuola Norm. Sup. Pisa (3)}, 17:189--206, 1963.

\bibitem{MR2763343}
D.~Mitrea and I.~Mitrea.
\newblock On the regularity of {G}reen functions in {L}ipschitz domains.
\newblock {\em Comm. Partial Differential Equations}, 36(2):304--327, 2011.

\bibitem{MR1501936}
C.B. Morrey, Jr.
\newblock On the solutions of quasi-linear elliptic partial differential
  equations.
\newblock {\em Trans. Amer. Math. Soc.}, 43(1):126--166, 1938.

\bibitem{MR1195131}
O.~A. Ole{\u\i}nik, A.~S. Shamaev, and G.~A. Yosifian.
\newblock {\em Mathematical problems in elasticity and homogenization},
  volume~26 of {\em Studies in Mathematics and its Applications}.
\newblock North-Holland Publishing Co., Amsterdam, 1992.

\bibitem{OB:2009}
K.A. Ott and R.M. Brown.
\newblock The mixed problem for the {L}aplacian in {L}ipschitz domains.
\newblock arXiv:0909.0061 [math.AP], 2009.

\bibitem{GS:1960}
G.~Stampacchia.
\newblock Problemi al contorno ellitici, con dati discontinui, dotati di
  soluzionie h\"olderiane.
\newblock {\em Ann. Mat. Pura Appl. (4)}, 51:1--37, 1960.

\bibitem{TOB:2011}
J.L. Taylor, K.A. Ott, and R.M. Brown.
\newblock The mixed problem in {L}ipschitz domains with general decompositions
  of the boundary.
\newblock To appear, Trans. Amer. Math. Soc.

\end{thebibliography}

\def\cprime{$'$} \def\cprime{$'$} \def\cprime{$'$} \def\cprime{$'$}

\noindent{\small \today}
\end{document}